\documentclass[preprint]{imsart}
\usepackage{textcase}

\usepackage{mathtools,amssymb,amsthm}
\usepackage{graphicx,color}
\usepackage{subcaption}
\usepackage{pifont}% http://ctan.org/pkg/pifont
\usepackage{booktabs}
\usepackage{dsfont}

\usepackage{natbib}
\setcitestyle{authoryear,open={(},close={)}} %Citation-related commands

\usepackage{float}
\usepackage[hyphens]{url}
\usepackage{bm}

\usepackage[toc,page]{appendix}
\usepackage{hyperref}
\usepackage[capitalize]{cleveref}
\usepackage{mathtools,amsmath,amsthm, amssymb}

\newcommand{\cC}{\mathcal{C}}

\newcommand{\cK}{\mathcal{K}}

\newcommand{\GG}{\mathbb{G}}

\newcommand{\NN}{\mathbb{N}}

\newcommand{\1}{\mathds{1}}

\newcommand{\kl}[2]{\KL(#1 \!\;\|\; \!#2)}

\newcommand*{\triplenorm}[1]{{\left\vert\kern-0.25ex\left\vert\kern-0.25ex\left\vert #1
    \right\vert\kern-0.25ex\right\vert\kern-0.25ex\right\vert}}

\DeclareMathOperator{\id}{id}

\newcommand{\R}{\mathbb{R}}

\renewcommand{\phi}{\varphi}
\newcommand{\eps}{\varepsilon}

\newcommand*{\E}{\mathbb E}

\newcommand*{\defeq}{\coloneqq}
\newcommand*{\rd}{\mathrm{d}}
\newcommand*{\dd}{\, \rd}

\DeclareMathOperator*{\KL}{KL}

\newcommand{\OT}{\text{OT}}

\usepackage{algorithm}
\usepackage{algpseudocode}
\usepackage[mathscr]{euscript}
\usepackage{bm}
\usepackage{xcolor}
\usepackage{enumitem}
\usepackage{comment}

\hypersetup{citecolor=cyan, colorlinks=true, linkcolor=blue!75!black}

\newtheorem{theo}{Theorem}
\theoremstyle{remark}
\newtheorem{prop}[theo]{Proposition}
\newtheorem{lem}[theo]{Lemma}

\newtheorem{rmk}[theo]{Remark}
\newtheorem{assum}[theo]{Assumption}
 \newcounter{prob}
\newtheorem{opProb}[prob]{Open Problem}

\startlocaldefs

\providecommand{\reals}{\mathbb{R}}
\providecommand{\diff }{\mathrm{d}}
\providecommand{\eps}{\varepsilon}
\providecommand{\KL}{\operatorname{KL}}

\providecommand{\id}{\operatorname{id}}

\providecommand{\LSE}{\operatorname{LSE}}
\providecommand{\dto}{\rightsquigarrow}
\providecommand{\oh}{\mathrm{o}}
\providecommand{\Oh}{\mathrm{O}}
\providecommand{\1}{\mathds{1}}
\providecommand{\interior}{\operatorname{int}}

\endlocaldefs

\begin{document}

\begin{frontmatter}

% "Title of the paper"
\title{The statistics of entropic optimal transport with decreasing regularisation}
\runtitle{EOT with decreasing regularisation}

% indicate corresponding author with \corref{}
 \author{\fnms{Gilles} \snm{Mordant}\thanks{gilles.mordant@yale.edu}}
 \address{%Department of Computational and Applied Mathematics\\ 
 Yale University \\
 %12 Hillhouse Avenue, 06511\\
 %New Haven, CT, USA.
 }
 \affiliation{Yale University}

%\author{\fnms{Gilles} \snm{Mordant}\ead[label=e1]{gilles.mordant@uclouvain.be}}
%\address{\printead{e1}}
%\and
%\author{\fnms{???} \snm{???}\ead[label=e2]{???}}
%\address{\printead{e2}}
%\affiliation{???}

\runauthor{G. Mordant}

\begin{abstract}
 We study the statistical properties of the entropic optimal (self) transport problem for smooth probability measures. 
We provide an accurate description of the limit distribution for entropic (self-)potentials and plans as the regularization parameter shrinks with the sample size; this regime is largely unexplored in the prior statistical literature, where $\epsilon$ is typically held fixed.

Additionally, we show that a rescaling of the barycentric projection of the empirical entropic optimal self-transport plans converges to the score function, a central object for diffusion models, and characterize the asymptotic fluctuations both pointwise and in $L^2$. 

Finally, we describe under what conditions the methods used enable to derive (pointwise) limiting distribution results for the empirical entropic optimal transport potentials in the case of two different measures and appropriately chosen shrinking regularization parameter. This endeavour requires a better understanding of the composition of Sinkhorn operators in the small $\eps$-limit, a result of independent interest.
\end{abstract}

%\begin{keyword}[class=MSC]
%\kwd[Primary ]{}
%\kwd{}
%\kwd[; secondary ]{}
%\end{keyword}

%\begin{keyword}
%\kwd{}
%\kwd{}
%\end{keyword}

\end{frontmatter}

%--------------------------------------------------------------------------------------
%                                     Introduction  
%--------------------------------------------------------------------------------------

\section{Introduction}
Optimal transport proved itself fundamental  in many areas of mathematics. Since its first appearance with the so-called Monge problem and its revitalization since Kantorovich's reformulation in the 1940's, many fields have embraced the theory, finding one way or another to bring it to fruition. 

In its primary form, for two Borel probability measures $\mu, \nu$ on $ K \subset \R^d$ that are absolutely continuous with respect to the Lebesgue measure,
 the problem involves finding an optimal transport map $T$ as the solution to 
 \[
  \inf_{T: T_\#\mu = \nu } \frac12 \int \|x - T(x)\|^2 \diff \mu(x), 
 \] 
 where $T_\#\mu$ is the pushforward measure, i.e., if $X\sim \mu$ then $T(X)\sim \nu$. Under suitable regularity assumptions (see Assumption~\ref{assum: Main} below), the optimal map $T_0$ exists, is unique and given by the gradient of a convex function \citep{Bre91}.

 This transport map is a central element that resonates far beyond optimal transport. In applications, the practitioners want to estimate $T_0$ based on data.
 In a precise statistical context, given two samples $X_1, \ldots, X_n \sim \mu$ and 
 $Y_1, \ldots, Y_n \sim \nu$, a fundamental question is to provide an estimator $\hat T_n$ of $T_0$.  Two requirements are that this estimator be computable in practise and exhibit good convergence rates. We discuss prior work on this question in Section~\ref{sec: Lit}.

The possibility to compute optimal transport efficiently at scale became clear with Cuturi's paper \citep{cuturi2013sinkhorn}. The idea there was to consider a relaxed version of the problem by adding an entropic regularization term. For a parameter $\eps > 0$, the \emph{entropic optimal transport} problem is given by
\begin{equation}
\label{eq:EOTPrimal}
     \OT_\eps(\mu,\nu):=   \min_{\pi \in \Gamma(\mu, \nu)} \iint  \tfrac12\|x-y \|^2 \diff  \pi (x,y)+ \eps \kl{\pi}{\mu\otimes\nu}\,,
\end{equation}
where $\Gamma(\mu, \nu)$ is the set of joint measures with marginals $\rho$ and $\mu$, and 
\begin{align*}
    \kl{\pi}{\mu\otimes\nu} = \iint \log \Bigl(\frac{\diff\pi}{\diff(\mu\otimes\nu)}\Bigr)\diff \pi\,,
\end{align*}
whenever $\pi$ has a density with respect to the product measure $\mu\otimes\nu$, and is ${+\infty}$ otherwise. Due to the strict convexity of the objective, a unique minimizer, the \emph{entropic optimal (transport) plan}, written $\pi_\eps$, always exists. The right way to think about $\pi_\eps$ is to see it as a \emph{blurred} version of $(\id, T_0)_{\#}\mu$ by a Gaussian smoothing at scale $\sqrt{\eps}$.

Entropic regularization was first introduced as an algorithmic tool to approximate the unregularized optimal transport cost, i.e., when $\eps = 0$, see \citet{PeyCut19} for a book-long exposition.   
Attention soon turned to its statistical properties at fixed $\eps$ (see Section~\ref{sec: Lit} below); the behavior as $\eps \to 0$ has also been studied from a geometric and large-deviations standpoint \citep{bernton2022entropic,nutz2022entropic,conforti2021formula}, and vanishing regularization has been exploited for map estimation by \citet{pooladian2021entropic}. Our focus is instead on distributional limits for the empirical potentials, plans, and their gradients in the regime $\eps_n \to 0$ with $n \to \infty$.

\subsection{Regularized optimal transport}

The primal formulation \eqref{eq:EOTPrimal} is not amenable to computation as it requires optimizing over all couplings. It does, however, admit a dual formulation which is extremely useful, both from an analytical and computational perspective. The dual formulation reads (see e.g., \cite{genevay2019entropy, nutz2021introduction})
\begin{align}\label{eq:EOTDual}
\nonumber
    \OT_\eps(\mu,\nu)=  \quad \sup_{(f,g) \in L^1(\mu\otimes\nu)}
   & \int f (x)\diff \mu(x) + \int g(y) \diff \nu(y) \\
    &\quad- \eps \iint e^{(f(x)+g(y)- \tfrac12 \|x-y\|^2)/\eps} \diff\mu(x)\diff\nu(y)+ \eps.
\end{align}
We call the maximizers in equation~\eqref{eq:EOTDual} the \emph{entropic optimal dual potentials}, written $(f_\eps, g_\eps)$. Note that they are unique up to constant shifts. Moreover, the entropic optimal plan can be readily expressed as a function of the entropic optimal potentials through a primal-dual recovery relationship \citep{Csi75}, i.e., 
\begin{align*}
    \diff\pi_\eps(x,y) = \exp\left(\frac{f_\eps(x)+g_\eps(y)-\tfrac12\|x-y\|^2}{\eps}\right)\diff\mu(x)\diff\nu(y)\,,
\end{align*}
which implies that, at optimality,
\(
    \OT_\eps(\mu,\nu) = \int f_\eps \diff\mu + \int g_\eps \diff\nu\,.
\)
Moreover, the entropic optimal dual potentials are known to satisfy the following relationships, see \cite{mena2019statistical,nutz2021entropic}.
\begin{equation}
\label{eq:dual_opt}
\begin{split}
\int e^{(f_\eps(x) + g_\eps(y) - \frac{1}{2}\|x-y\|^2)/\eps} \diff \mu(x) & =1,  \forall y \in K,\\
\int e^{(f_\eps(x) + g_\eps(y) - \tfrac12\|x-y\|^2)/\eps} \diff \nu(y) & =1,  \quad \forall x \in K\,.
\end{split}
\end{equation}
These relations are crucial for the rest of our work and the system of equations above is called the Schrödinger system.

The \emph{entropic map} between $\mu$ and $\nu$  is the barycentric projection of $\pi_\eps$, which, for $x \in K$, is given by
\begin{equation}
 T_\eps(x):=  \E_{\pi_{\eps}}[Y \mid X = x] = \int y \diff \pi_{\eps}^x(y).
\end{equation} 
Alternatively, by equation~\eqref{eq:dual_opt},
\begin{equation}
\label{eq:entBar}
T_\eps(x) = \frac{\int y e^{\frac 1 \eps(g_\eps(y) - \frac{1}{2}\|x-y\|^2)} \diff \nu(y)}{\int e^{\frac 1 \eps( g_\eps(y) - \frac{1}{2}\|x-y\|^2)} \diff \nu(y)}\,.
\end{equation}
It can also be verified from these optimality conditions that  $T_\eps = \text{id} - \nabla f_\eps$ as shown in \citet[Proposition 2]{pooladian2021entropic}.

A particularly convenient feature of the regularized approach is that, defining $\hat f_\eps, \hat g_\eps$ to be the empirical counterparts of  $f_\eps,  g_\eps$ obtained from $\mu_n , \nu_n$ (the empirical measures built from the samples), the estimator 
\begin{equation}
\label{eq: Estim}
\hat T_\eps(x) \defeq \frac{\int y e^{\frac 1 \eps(\hat g_\eps(y) - \frac{1}{2}\|x-y\|^2)} \diff \nu_n(y)}{\int e^{\frac 1 \eps( \hat g_\eps(y) - \frac{1}{2}\|x-y\|^2)} \diff \nu_n(y)}\,.
\end{equation}
admits a neat form which look like a tilted kernel density estimator. 
This raises the main question of this work: 
\begin{center}
\textit{How should $\eps$ be chosen as a function of $n$, and  under what conditions does $ \hat T_{\eps_n}(x)$ admit a central limit theorem? }
\end{center}
Before discussing what we can achieve, let us first focus on the self-transport case, whose analysis paves the way to understanding the two-measure case. Furthermore, the self-transport case reveals beautiful connections to diffusions.

\subsection{Estimation of scores via self-transport}
\label{se: scoresDetour}

In the context above, what happens when $\mu=\nu$? 
Although mapping a measure to itself is trivial in the unregularized case, the regularized problem is far from vacuous: the entropic self-transport potentials encode geometric information about $\mu$, and Sinkhorn iteration on $(\mu, \mu)$ is closely related to a diffusion.
This connection has been recognised and exploited in recent works, see for instance \citet{sander2022sinkformers,gottwald2025stable}. 

Understanding what happens in this case will illuminate the analysis of the general two-measure problem. When $\mu =\nu$, symmetry forces $g_\eps\equiv f_\eps$, and the Schrödinger system \eqref{eq:EOTDual} collapses to a single equation. Setting   \(  u_\eps(x):= \exp(-f_\eps(x)/\eps)\)
this equation becomes the fixed point relation
\begin{equation}
\label{eq:fixPtEq}
u_\eps(x) = \int 
    \frac1{u_\eps(y)}\exp\left(-\tfrac1{2\eps} \|x-y\|^2\right)
    \diff \mu(y).
\end{equation}
\noindent
A heuristic Fourier argument (see Appendix~\ref{sec: AddRes}) for $\eps \to 0$ suggests that 
\(
-2f_\eps(x)/\eps = \log(\mu) + d \log (2\pi\eps)/2+\oh(1)
\)
which is known to hold in the Gaussian case \citet[Proposition 4]{pooladian2022debiaser} (see also Appendix~\ref{sec: GausEOT}) and was established for certain log-concave measures in \citet{agarwal2026langevin}. This expansion shall be established in Proposition~\ref{prop:secOrdExp}.
Taking the gradient on both sides, one sees the score appearing, which motivates the use of an empirical counterpart in the case where one only has access to a sample.

Thus, anticipating the analysis below, this motivates the score estimator 
\begin{align}
\label{eq: estimator_def}
    \hat{s}_{\eps}(x) \defeq -\frac{2}{\eps}\nabla \hat f_{\eps}(x) = \frac{2}{\eps}\left(x - \frac{\sum_{i=1}^nX_ie^{\hat f_\eps(X_i)/\eps} e^{-\tfrac{1}{2\eps}\|x-X_i\|^2}}{\sum_{j=1}^ne^{\hat f_\eps(X_j)/\eps} e^{-\tfrac{1}{2\eps}\|x-X_j\|^2}}\right)\,,
\end{align}
where $X_1, \ldots, X_n \sim \mu$ are i.i.d. %The second equality is the discrete counterpart to Equation \eqref{eq:entBar} below.

Score functions are the central object of diffusion-based generative models, which makes the statistical behavior of \eqref{eq: estimator_def} a question of independent interest.
\begin{center}
\textit{What are the statistical guarantees of the self-transport score estimator?}
\end{center}

The minimax rates established by \citet{wibisono2022convergence} are the natural standard we want to compare with.

\subsection{Main results (informal)}

%We will conflate the measures with their densities with respect to the Lebesgue measure.

\begin{assum}
\label{assum: Main}
Throughout, we assume the following.
\begin{enumerate}
 \item  $K$ is a compact convex subset of $\reals^d$, whose boundary $\partial K$ is $\cC^2$. 
    \item \label{ass: A} The densities   $\mu,\nu$ are supported on $K$ and are $\cC^4$. Furthermore, we have that $\ell \le \mu(x), \nu(y) \le L$, for $0<\ell \le  L$. 
  \end{enumerate}
  \end{assum}

The part of Assumption~\ref{assum: Main} pertaining to the support and boundedness of the densities is standard, it is made in works such as \cite{10.1214/20-AOS1997,pooladian2021entropic,manole2024plugin}, among others. 

We state two mains results informally. First, under regularity conditions, for fixed $x$ and choosing $ n^{-2/(d+2)} \ll \eps_n \ll n^{-2/(d+4)}$, it holds that
\[
\sqrt{n} \eps_n^{d/4-1/2}\big(\hat T_{\eps_n} (x)  - T_0(x) \big ) \rightsquigarrow \mathcal{N}( 0, \Sigma_1(x)), 
\]
for $ n\to \infty$,  
where 
\[
\Sigma_1(x) =    \lim_{n \to \infty}\frac1{\mu(x)} \int_{\reals^d} \left( \frac{
 \nabla^2 \phi(x) \xi \xi^\top \nabla^2 \phi(x) \mathcal{F}\left[ e^{-\frac1{\eps_n} D(x,\cdot)   }\right](\xi)}
{(2\pi\eps_n)^{d/4}  \xi^\top \nabla^2 \phi(x) \xi }\right)^2 \diff \xi.
\]  

This shows that $\hat T_{\eps_n} (x)$ is asymptotically normal at a rate slower than the parametric $\sqrt{n}$ by a factor depending on the dimension, the exact statement is Theorem~\ref{thm: twoMeasures} below.
Fundamentally, the result builds on our understanding of the composition of Sinkhorn operators for decreasing regularization. This result identifies the operator that drives the fluctuations of the empirical potentials in the two-measure case, and reveals a striking parallel with the linearization of the Monge--Ampère equation developed in \cite{manole2024plugin}, see Appendix~\ref{sec: Manole} for an overview.

\noindent
To our knowledge, this is the sharpest pointwise description of the fluctuations of entropy-regularized barycentric maps around the unregularized transport map currently available. 

Our second result concerns the score estimator defined in \eqref{eq: estimator_def}.
Upon choosing  $ n^{-2/(d+2)} \ll \eps_n \ll n^{-2/(d+6)}$, we prove 
\[
\sqrt{n} \eps_n^{d/4+1/2}\big( \hat{s}_{\eps_n}(x) - \nabla \log \mu (x)\big) \rightsquigarrow \mathcal{N}(0, \Sigma_2(x)), 
\]
as $n\to \infty$ where $\Sigma_2(x)$ is some covariance matrix. This is the key message of Theorem~\ref{thm: Plans}.

We also establish $L^2(\mu)$ rates for $\hat{s}_{\eps_n}$ (Theorem~\ref{thm:main_stat_result}), since this is the natural metric for applications to diffusion models.

\subsection{Related literature}
\label{sec: Lit}

 The literature on optimal transport is by now so vast that any complete overview is impossible; we restrict ourselves to works directly relevant to the present paper, organised around our three contributions.

\subsubsection{Limit theorems for (entropic) optimal transport}

A substantial body of work has studied distributional limits and convergence rates for variants of the optimal transport problem. Listing in detail the substance of these contributions would take us too far afield; we warmly invite the reader to consult the following references and the works they cite. For the cost, see 
\cite{sommerfeld2018inference,del2019central,hundrieser2022unifying,del2023improved,hundrieser2023empiricaloptimaltransportestimated, manole2024sharp,goldfeld2024limit,goldfeld2024statistical}.
For the plan, \cite{klatt2022limit,manole2023central,gonzalez2022weak, goldfeld2024limit}  are relevant references. For an overview of the broader statistical theory of entropic optimal transport, the reader can consult  \citet{chewi2024statistical}.

The work of \cite{gonzalez2022weak} deserves special mention: to our knowledge, it is the only prior reference providing a clear, detailed description of the asymptotic fluctuations of empirical entropic potentials in terms of Sinkhorn operators. Our Theorem~\ref{thm: twoMeasures}  extends this picture to the regime $\eps_n  \to 0$, where the operators vary with $\eps_n$ and a more delicate analysis is required. More generally, all of the works cited above operate either in the unregularized setting or at fixed $\eps >0$; the regime $\eps_n  \to 0$ with $n  \to \infty$, which is the focus of the present paper, has remained largely unaddressed. 

Concurrently, \citet{dou2024optimal} and \citet{zhang2024minimax} have obtained minimax-optimal score estimators under milder smoothness assumptions. In particular, Lipschitz continuity of the score suffices, whereas our Theorem 11 requires one additional derivative. Classical work on kernel density estimation has further shown that KDE-based estimators can adapt to intrinsic dimension \citep{kim2019uniform,jiang2017uniform}. \citet{groppe2023lower} have shown that the empirical EOT cost obeys a lower complexity adaptation principle, with rates depending on the simpler of the two measures rather than the ambient dimension. Whether a similar phenomenon holds for the potentials or their gradients remains open; our numerical experiments in Section 6.2 suggest it does not.

Other relevant works on the small-$\eps$ behavior of entropic optimal transport include \citet{mena2019statistical,nutz2022entropic,pal2024difference,bernton2022entropic,conforti2021formula}.
In all that we do, the only cost of interest is the quadratic one, as the theory is particularly neat in this case. Just like in  \citet{pal2024difference}, the approach should work for more general costs provided that the Bregman divergence based on the optimal transport potentials (the analog of $D$ in Section 2) is sufficiently well behaved. The interest for more general (and even estimated) costs is clearly important for recent developments in OT; see \citet{grave2019unsupervised,pooladian2024neural,hundrieser2024empirical}.

\subsubsection{Estimation of transport maps} The statistical estimation of the Brenier map has received much attention in recent years. \cite{hutter2021minimax} established minimax rates for smooth transport maps. The entropic map was subsequently analysed by \cite{pooladian2021entropic,pooladian2023minimax} as a computationally tractable estimator of the optimal transport map for the quadratic cost, building on \cite{hutter2021minimax}. The same map was studied in \citet{rigollet2025sample} using a sample-complexity argument. \citet{manole2024plugin} propose a different estimator, based on plugging a kernel density estimator into the Monge--Ampère equation, and prove central limit theorems for it; we sketch their approach in Appendix~\ref{sec: Manole} because of the beautiful parallels with our Theorem 16. A completely different approach, based on a sum-of-squares reformulation of the semi-dual formulation, is provided in \cite{vacher2024optimal}; their method is both computationally attractive and nearly achieves the minimax rate (up to logarithmic factors).
All of the above estimators come either with finite-sample rates at fixed regularization (or no regularization at all), or with central limit theorems whose centering and scaling depend on the regularization being held fixed. Theorem 16 provides, to our knowledge, the first central limit theorem for the entropic map $\hat T_{\eps_n}$ centered at the unregularized Brenier map $T_0$ and with $\eps_n\to 0$.

\subsubsection{Self-transport, Sinkhorn divergence, and score estimation} 

The entropic optimal self-transport problem has been studied independently by, e.g., \cite{genevay2019sample,feydy2019interpolating} in general-purpose machine learning, and by  \cite{marshall2019manifold}  in the context of manifold learning. The Laplace-method calculation of  \cite{marshall2019manifold}  is the starting point for our second-order expansion in Section 3.1.
In machine learning, the entropic self-OT problem appears most prominently in the definition of the Sinkhorn divergence,
\begin{align*}
    S_\eps(\mu,\nu)\defeq \OT_\eps(\mu,\nu) - \frac12\bigl(\OT_\eps(\nu,\nu) + \OT_\eps(\mu,\mu)\bigr)\,,
\end{align*}
which satisfies $S_\eps(\mu,\nu) \geq 0$ with equality if and only if $\mu=\nu$ \cite[Theorem~1]{feydy2019interpolating}.  $S_\eps(\rho,\mu)$ is a \emph{debiased} version of $\OT_\eps(\mu,\nu)$,  removing the leading-order entropic bias. This will become particularly apparent in our small epsilon  expansions. Statistical and algorithmic implications of the Sinkhorn divergence were considered in \cite{feydy2019interpolating,pooladian2022debiaser}, among others. \citet{marshall2019manifold} proposed to use self-transport to perform manifold learning, which was further extended to the empirical case by \citet{landa2021doubly}.  In a similar vein, the work of \citet{landa2023robust} used self-transport as a robust tool for density estimation.

Finally, we note that the Kullback--Leibler regularization is not the only interesting one. Optimal transport with $L^2$ regularization has been considered by \citet{blondel2018smooth}, who noticed that it gives sparse transport plans. Building on this, \citet{zhang2023manifold} proposed to use to carry out manifold learning relying on $L^2$ regularized self-transport; the approach is robust to heteroskedastic noise in high-dimensional spaces.
 Further developments on quadratically regularized OT include  \citet{garriz2024infinitesimal,gonzalez2025sparse,gonzalez2025sample,liu2025beyond}.

 For a probability measure $\rho\propto e^{-V}$, the estimation of score functions has, classically, received less attention than the closely related  estimation of densities or gradient of densities, which is a standard statistical topic \citep{silverman2018density,Tsy09}.  
A procedure to estimate log-concave densities is 
explained in the  review paper by \citet{samworth2018recent}. The methodology therein relies on the so-called log-concave projection;  an optimisation problem that is quite intricate to solve in higher dimensions. Even though minimax rates are known in that case, they only apply to the estimation of the density and not its gradient.  \citet{seregin2010nonparametric} obtain $n^{-2/(d+4)}$ as pointwise minimax rate with respect to the absolute loss function at a point of sufficient regularity.
Still, the global picture is far from complete.  
Another view on the problem could be that of estimation of the gradient of a shape constrained function. To the best of our knowledge, that approach hasn't been pursued in general dimensions.

A major line of research follows the score-matching framework of \citet{hyvarinen2005estimation} in which the score is estimated within a parametric class.  \citet{sriperumbudur2017density}  extended it to a reproducing kernel Hilbert space (RKHS) framework and a further extension was proposed by \citet{zhou2020nonparametric}.
More recently, \cite{song2020sliced} proposed a sliced version of score matching. Statistical efficiency of score matching under functional inequalities was studied by \citet{koehler2022statistical}. While completing this work, we became aware of the paper by \citet{wibisono2024optimal}, in which the authors propose an estimator of the score based on (truncated) kernel density estimation (KDE). By deriving suitable lower bounds, they show that their estimator is minimax optimal up to logarithmic factors. Their estimator and analysis are substantially different from ours: we use entropic regularization rather than KDE truncation, and our analysis goes through the linearization of the Schrödinger fixed-point equation rather than through bias-variance decomposition of a kernel estimator.

\subsection{Outline of the paper}

 Section~\ref{sec: Background} introduces the Bregman divergence $D$ associated to the Kantorovich potentials, the central geometric object in all subsequent expansions. Section \ref{sec:main_limit} derives small-$\eps$ expansions of the entropic potentials in the self-transport case (Proposition~\ref{prop:secOrdExp}) and in the two-measure case (Propositions~\ref{prop: TwoMeasFixPt} and \ref{prop: otBregmanExp}), the latter resting on a WKB ansatz. Section~\ref{sec: CLT} contains the main central limit theorems: Theorems~\ref{thm: multivariate_limit} and \ref{thm: Plans} in the self-transport setting and Theorem~\ref{thm: twoMeasures} in the two-measure setting. The two-measure CLT is built on a linearization of the empirical fixed-point system (Proposition \ref{prop: Inversion}) and a refined description of the composition of Sinkhorn operators at small $\eps$ (Theorem~\ref{thm: LimOp}), which is of independent interest. Section~\ref{sec: AppliScore} translates the pointwise analysis into $L^2(\mu)$ rates for the score estimator (Theorem~\ref{thm:main_stat_result}). Section~\ref{sec:Numerics} reports numerical experiments that confirm the limiting variance of Theorem~\ref{thm: multivariate_limit}  and probe possible intrinsic-dimension adaptation. Section~\ref{sec: Conc} concludes with three open problems. All proofs are deferred to the Appendix.

\subsection{Notation}

In the sequel, all random variables will be defined on a common, sufficiently rich probability space $(\Omega, \mathcal{A}, \mathrm{P})$.  
We will slightly abuse notation and use both $\mu$ as a probability measure and to mean its density $\diff \mu/ \diff \lambda$, where $\lambda$ is the Lebesgue measure. 
 For a probability measure $\mu$, we write $ \mu_n$ for the empirical measure of $\mu$ based on $n$ i.i.d.\ samples. 
 Convergence in distribution is denoted by $\dto$. 
As usual, $a \lesssim b $, means that there exists a constant $C>0$ such that $a \le Cb$. We sometimes use $ a_n\asymp b_n$ for asymptotic equivalence as $n\to \infty$. We  set $\mathbb{G}_n = \sqrt{n}(\mu_n- \mu)$. For a function $f \in L^1(\mu)$, we use $\mu f \defeq \int f \dd \mu$. 
$\mathcal{F}$ denotes the Fourier transform operator.
The Fenchel--Legendre transform of a function $\phi(y)$ defined on a domain $\mathcal{Y}$ is $\phi^*(x):= \sup_{y\in \mathcal{Y}}( \langle x, y \rangle - \phi(y) )$. The set of $k$ times continuously differentiable functions is denoted $\cC^k$. The Hessian of a function $f$ at a point $x$ will be denoted $\nabla^2f(x)$. 
Finally, we set 
\[
k_\eps(x,y) := \frac{\exp\left( -\tfrac1{2\eps} \|x-y \|^2\right)}{ (2\pi\eps)^{d/2}}.
\]
 
%----------------------------------------------------------------------------------------------------
%                                     Background
%----------------------------------------------------------------------------------------------------

\section{Background on optimal transport}
\label{sec: Background}

Let us turn to the population unregularized problem. We refer the reader to \citet{Vil08,San15,villani2021topics} for in-depth introduction to the unregularized case.
Note that the unregularized problem also admits a dual formulation; denote a pair of dual optimizers by $(f_0, g_0)$.
Because of Assumption~\ref{assum: Main}, Caffarelli's regularity theory \citep{Caffarelli1996Bound2, Urbas1997} implies that the potentials are 
$\mathcal{C}^4,\alpha$. 
Let $\varphi := \|\cdot\|^2/2 -f_0  $ and $\psi := \|\cdot\|^2/2 -g_0 $ be convex potentials.
Let 
\[
D(x,y) := \varphi(x) + \psi(y) - \langle x, y \rangle, 
\]
be the Fenchel--Young gap or Bregman divergence associated to optimal transport; \emph{it is the central geometric object governing all our expansions.}
It is a Bregman divergence\footnote{For $F$ a convex, continuously differentiable function on a convex set $K\subset \reals^d$, the Bregman divergence associated with $F$ between two points $p,q$ is given by $F(p)- F(q) -\langle \nabla F(q), p-q\rangle$. }, as, by duality, $\psi$ is the Fenchel--Legendre transform of $\varphi$, i.e., 
\[
\psi(y) = \sup_{x \in \reals^d} \big( \langle x, y \rangle - \varphi(x)\big). 
\]
Owing to smoothness, this implies that 
\(
\psi(y) = \langle [\nabla \varphi]^{-1}(y), y \rangle - \varphi \big( [\nabla \varphi]^{-1}(y)\big).
\)
From which we write, for $z= [\nabla \varphi]^{-1}(y)$,
\begin{equation}
D(x,y) = \varphi(x) - \varphi \big( z \big) - \langle \nabla \varphi (z) , x-z \rangle.
\end{equation}
An important fact of Bregman divergences is that they are \emph{dually flat}, i.e., 
\begin{equation}
\label{eq: DualFlat}
\nabla^2 \psi \circ \nabla \varphi = \big[\nabla^2 \varphi\big]^{-1}.
\end{equation}
This identity will be used repeatedly to convert between $x-$ and $y-$derivatives.

In the sequel, we will denote by $x^*$ the image of the point $x$ by the optimal transport map\footnote{The latter exists owing to our assumptions on the densities.}, i.e., $x^*:=T_0(x)$.
A key consequence is the local quadraticity of $D$; for $z$ near $x$,
\begin{equation}
\label{eq: LocalQuad}
D(z,\, x^*) = \tfrac{1}{2}(z - x)^\top \nabla^2\varphi(x)\,(z-x) + o\!\left(\|z - x\|^2\right),
\end{equation}
where we used that $\nabla\varphi(x) = x^*$ cancels the linear term in the Taylor expansion of $\varphi(z)$, and that $D(x, x^*) = 0$. Thus $D(\,\cdot\,, x^*)$ is locally equivalent to the Riemannian metric induced by $\nabla^2\varphi(x)$, i.e., the Hessian of the Kantorovich potential.

%----------------------------------------------------------------------------------------------------
%                                     Small epsilon development
%----------------------------------------------------------------------------------------------------

\newpage
\section{Fixed point equations and small $\epsilon$ developments}
\label{sec:main_limit}

\subsection{One measure case}
\label{sec: ExpOneMeasure}

As already announced, we start from equation~\eqref{eq:fixPtEq}. Let us first prove the following pointwise result which is inspired by the work of \citet{marshall2019manifold}.

\begin{assum}
\label{assum: seq}
Let $\mu$ be a measure such that sequence of rescaled self potentials 
\[
\left(  \frac{\exp(  f_\eps/ \eps) }{(2 \pi \eps)^{d/4}}\right)_\eps 
\]
is bounded in $\cC^4$.
\end{assum}

\begin{rmk}[Scope of Assumption 2 and related literature] Assumption 2 is a uniform-in-$\eps$ bound on the rescaled self-potentials, needed in Proposition~\ref{prop:secOrdExp} to differentiate the expansion to second order. It is implicit  in \citet{marshall2019manifold}.
 We prove a weaker statement without derivative control in Proposition~\ref{prop: PotSingMeasL2}.
 
Assumption~2 holds in the Gaussian case.
Further,  for $\cC^6$ distributions supported on compact manifold of dimension at most 5, a similar expansion is known to hold, see \citet[Theorem~2.2, Eq.~(11)]{landa2023robust}. 
In a slightly different setting,  \citet[Lemma~4.6]{deb2023wasserstein}  derive an analogous expansion from log-concavity-type hypotheses relying on a Laplace Method.
\end{rmk}

\begin{prop}[Second-order expansion]
\label{prop:secOrdExp}
Under Assumption~\ref{assum: seq}, for any fixed point $x\in \interior(K)$, 
as $\eps \to 0$, the following holds.
\begin{align}\label{eq:secOrdExp}
     \nonumber
        \exp\left(\tfrac{-2f_\eps(x)}{\eps}\right)  
    = \mu(x) (2\pi\eps)^{d/2}   \left( 1 + \eps  \left(  \operatorname{tr} \tfrac{\nabla^2\mu(x)}{4\mu(x)} - \tfrac{\|\nabla \log \mu(x) \|^2}{8} \right)+ \oh(\eps)\right)\,.
\end{align}
\end{prop}

The proof is provided in the Appendix~\ref{sec:secOrdExp}.
\begin{rmk}[Leading bias and harmonic characteristic]
\label{rmk: harmChar}
As noted in \citet{agarwal2024iterated}, the term of order $\eps$ in the expansion is the harmonic characteristic and can be interpreted as a ``mean acceleration'' of a Langevin bridge. 
The second order expansion puts on solid ground the heuristic Fourier computation hinted at in the introduction that motivated the score estimation application. It also gives a glimpse into the links between regularization and diffusion. 
\end{rmk}

\begin{prop} 
\label{prop: PotSingMeasL2}
Under Assumption~\ref{assum: Main}, it holds 
\[
\lim_{\eps\to 0} \left \| - 2 f_\eps/\eps - \log \big( \mu (2 \pi \eps)^{d/2} \big) \right \|_{L^2 (\mu)} = 0. 
\]
\end{prop}
The proof is in 
Appendix~\ref{sec: PotSingMeasL2}

\begin{rmk}[Expansion in $L^\infty$]
The expansion above is not fine enough for a stronger topology. The correct form is $ \log \big( \mu (2 \pi \eps)^{d/2} \omega_{\eps, K} \big) $ where $\omega_{\eps, K}$ is a sequence of functions adjusting for the fact that the Gaussian kernel doesn't integrate to 1 if the point of evaluation is too close to the boundary.  
\end{rmk}

\subsection{Two-measure case and Léger--Vialard expansions}

To get results as in the self-transport case, a first fundamental question is whether there exists the same type of development as  for the potentials $f_\eps, g_\eps$ as  $\eps\to 0$. The analytics foundations were laid in  \citet{leger2023geometric} and such expansions have been announced at a BIRS Workshop by \citet{Flav}. Remark also that this type of developments further holds true in the Gaussian case, see \citet{mallasto2022entropy}.
Proposition~\ref{prop: TwoMeasFixPt} below records the explicit second-order form in our setting.

Let us shift the optimal potential of \eqref{eq:dual_opt} by some well chosen quantities and doing so introduce  $\tilde  f_\eps,  \tilde  g_\eps$ as follows
\begin{align}
\label{eq: ExpTwoMeas}
\nonumber
 f_\eps(x)  &= \tilde  f_\eps(x) +  f_0(x) - \frac{\eps}{2} \log   \mu(x) - \frac{\eps d}4\log (2\pi \eps) \\
%-----------
g_\eps(y) &= \tilde  g_\eps(y) + g_0(y) - \frac{\eps}{2} \log   \nu(y) - \frac{\eps d}4\log (2\pi \eps) .
\end{align}
We will see that $\tilde  f_\eps,  \tilde  g_\eps$ will be negligible in a suitable sense when $\eps \to 0$.

After that change of variable, the behavior of the pair of optimal functions $\tilde  f_\eps(x), \tilde  g_\eps(y)$ can be recovered\footnote{Mathematically, not numerically.} from the Schrödinger system.
\begin{prop}
\label{prop: TwoMeasFixPt}
Recall the definition in equation~\eqref{eq: ExpTwoMeas}. It holds that 
\label{eq: TwoMeasFixPt}
\[
\lim_{\eps \to 0} \ \Big \|\tilde f_\eps(x)/\eps  \Big \|_{L^2(\mu)} = 0 \quad \text{and} \quad \lim_{\eps \to 0} \   \Big  \|\tilde g_\eps(x)/\eps  \Big \|_{L^2(\nu)} = 0. 
\]
\end{prop}
The proof is deferred to Appendix~\ref{sec: TwoMeasFixPt}.

\begin{rmk}[Sinkhorn divergence and debiased map estimation]
\label{rem:debias}
The expansions of Proposition~\ref{prop:secOrdExp} and equation~\eqref{eq: ExpTwoMeas} help inform the Wasserstein gradient of the Sinkhorn divergence~\citep{feydy2019interpolating}. Decorating the potentials with the measures the problems involve,  it is true that  $\nabla_{\!W}S_\eps(\mu, \nu) = \nabla f^{\mu \leftrightarrow \nu}_\eps - \nabla f^{\mu \leftrightarrow \mu}_\eps$. The leading $\eps$-bias of each potential is the same score correction $-\tfrac{1}{2}\nabla\log\mu$, so the two cancel exactly: at the population level, the Sinkhorn-divergence gradient recovers $\nabla f_0$ with bias $o(\eps)$, an order of magnitude better than the entropic potential alone, which has bias $O(\eps)$. This suggests that a debiased map estimator built from the Sinkhorn divergence should outperform the barycentric map $T_{\eps_n}$ in the small-$\eps$ regime. The empirical picture is more delicate. \citet{pooladian2022debiaser} show that this debiasing benefit is genuine when $\eps$ is small and the sample size is large, but can be statistically detrimental in the opposite regime. The lower complexity adaptation enjoyed by the unregularised cost~\citep{groppe2023lower} does not extend to the empirical Sinkhorn divergence, which pays the full ambient-dimensional sample-complexity price~\citep{rigollet2025sample}. \end{rmk}

\subsection{Second order expansion in the two-measure case}
\label{sec: ExpTwoMeasures}
In the self-transport case, Proposition~\ref{prop:secOrdExp} allowed (under some assumptions) to get to a second order expansion, relying on the Laplace method.

To derive such a result, we make a WKB \emph{ansatz}%
\footnote{This type of assumption is known as the WKB approximation in 
physics, a general method that can be applied to the Schrödinger equation 
for instance. The first mathematician to set it on solid ground was 
\citet{WKB}.}, i.e., we require that
\[
	\tilde{f}_\eps(x) + \tilde{g}_\eps(y) 
=: \eps\check{f}(x) + \eps\check{g}(y) + o(\eps).
\]
Based on the fixed-point equations above, we get the system
\begin{align*}
1 =   \int    \frac{e^{(\tilde f (x) + \tilde g(y) -D(x,y))/\eps}}{(2\pi \eps)^{d/2} }\sqrt{\tfrac{\nu(y) }{\mu(x)}} \diff y \\
1 =   \int    \frac{e^{(\tilde f (x) + \tilde g(y) -D(x,y))/\eps}}{(2\pi \eps)^{d/2}} \sqrt{\tfrac{\mu(x) }{\nu(y)} } \diff x.
\end{align*}

\begin{prop}[Second order expansion in the two-measure case]
\label{prop: otBregmanExp}
In the setting above,  for $y$ in the interior of $\operatorname{supp} \nu$, set $x_0 := T^{-1}(y)$ and $A(y) := [\nabla^{2} \phi(x_0)]^{-1}$. Then as $\eps \to 0$,
\begin{align*}
&  \check{f}(T^{-1}(y)) + \check{g}(y)\\
  &= -\frac{1}{2 \sqrt{\mu(x_0)}} \operatorname{tr}\!\bigl( A(y) \nabla^{2}\!\sqrt{\mu}(x_0) \bigr)
     + \frac{1}{2 \mu(x_0)} \nabla\!\sqrt{\mu}(x_0)^{\top} A(y) \nabla\!\sqrt{\mu}(x_0) \\
  &\quad
    + \frac{1}{6 \sqrt{\mu(x_0)}} \sum_{i j k \ell}
      \phi_{ijk}(x_0) \partial_\ell \sqrt{\mu}(x_0)
      \bigl( A_{ij} A_{k\ell} + A_{ik} A_{j\ell} + A_{i\ell} A_{jk} \bigr) \\
  &\quad
    + \frac{1}{24} \sum_{i j k \ell}
      \phi_{ijk\ell}(x_0)
      \bigl( A_{ij} A_{k\ell} + A_{ik} A_{j\ell} + A_{i\ell} A_{jk} \bigr) \\
  &\quad
    - \frac{1}{72} \sum_{\substack{i j k \\ i' j' k'}}
      \phi_{ijk}(x_0) \phi_{i'j'k'}(x_0)
      \sum A_{\sigma_1 \sigma_2} A_{\sigma_3 \sigma_4} A_{\sigma_5 \sigma_6}
    + o(1),
\end{align*}
where the last sum ranges over the $15$ perfect matchings of $\{i, j, k, i', j', k'\}$ into three pairs, and $A_{ij} = A_{ij}(y)$. The symmetric identity for $\check{f}(x) + \check{g}(T(x))$, obtained by interchanging $(\mu, \phi)$ and $(\nu, \phi^{*})$ in the first equation, together with the above determines $\check{f}$ and $\check{g}$ up to an additive constant.
\end{prop}

The proof is provided in Appendix~\ref{sec:otBregmanExp}.
With the expansions of Sections~\ref{sec: ExpOneMeasure} and \ref{sec: ExpTwoMeasures} in hand, the entropic potentials are understood to the $\Oh(\eps)$ bias. In the next section we turn to the fluctuations around these expansions and establish the main central limit theorems.

%----------------------------------------------------------------------------------------------------
%                                     Results
%----------------------------------------------------------------------------------------------------

\section{Central limit theorems}
\label{sec: CLT}
We now state the main distributional results of the paper: pointwise central limit theorems for the empirical entropic potentials and their gradients in the regime  $\eps_n \to 0
$ as $n \to \infty$.

 Section~\ref{sec: OneMesCLTs} treats the self-transport case, where the analysis is cleanest and where the score-estimation application of Section~\ref{sec: AppliScore} is built. Section~\ref{sec: TwoMeas} treats the general two-measure case. The heavier machinery required there, namely the linearization of the empirical Schrödinger system and an asymptotic description of the composition of Sinkhorn operators, is collected in Proposition~\ref{prop: Inversion}, Theorem~\ref{thm: LimOp}, and Proposition~\ref{prop: SinkFluctKer} before the final CLT, Theorem~\ref{thm: twoMeasures}, is stated.

\subsection{One measure case}
\label{sec: OneMesCLTs}
Central limit theorems for entropic potentials at fixed $\eps$ were established by \citet{gonzalez2022weak} and \cite{goldfeld2022limit}. 
Our proof proceeds by a Taylor linearization of the self-transport fixed-point functional around the population potential, followed by inversion of $\id +K_{\eps_n}$ and an application of the Lindeberg–Feller theorem. The main technical difficulty absent from the fixed- $\eps$
setting is that the linearization requires a fine control on  the empirical fluctuations, the remainder, and the resolvent itself as they each carry their $\eps_n$-dependent scale.  
The limiting variance differs from the fixed-$\eps$ case: because the kernel $k_{\eps_n}$
localizes as $\eps_n \to 0$, evaluations at distinct points are asymptotically independent, in contrast with the fixed-$\eps$ regime where the kernel maintains long-range correlations.

\begin{theo}[Limit distribution for the empirical potentials]
\label{thm: multivariate_limit}
    Consider an i.i.d.\ sample $X_1, \ldots, X_n \sim \mu$ as above and suppose that Assumption~\ref{assum: Main} holds. Fix $m\in \NN$ and pairwise distinct  points $x_1, \ldots, x_m \in \interior(K)$. Then, as $n \to \infty$,
\begin{multline*}
\sqrt{n}\, \eps_n^{d/4}\begin{pmatrix}
\frac{ \hat{f}_{\eps_n}(x_1) - f_{\eps_n}(x_1)}{\eps_n}\\
\vdots \\
\frac{\hat{f}_{\eps_n}(x_m) - f_{\eps_n}(x_m) }{\eps_n}
\end{pmatrix} \\
\dto \mathcal{N}\left(0_m, \frac{C_3}{4 (2\pi)^{d/2}} \operatorname{diag}  \Big( \mu(x_1)^{-1}, \ldots,\mu(x_m)^{-1}\Big)  \right),
\end{multline*}
 provided that 
$
\sqrt{n} \eps_n^{d/4}/\sqrt{\log n}   \to \infty$, 
where 
\[
C_3 :=  \sum_{0\le \kappa,\kappa' \le \infty } 2^{- \kappa-\kappa'} \sum_{\eta=0}^\kappa\sum_{\eta'=0}^{\kappa'} {\kappa \choose  \eta} ( -1)^\eta ( -1)^{\eta'} {\kappa' \choose  \eta'} 
( \eta + \eta' +2)^{-d/2}.
\]
\end{theo}
The proof of that theorem is in Appendix~\ref{sec: multivariate_limit}. We further provide additional elements on the constant $C_3$ in Appendix~\ref{sec: RomanNumbers}.
The diagonal structure arises because the fluctuations are determined, up to negligible terms, by the behavior of the kernel in a $\sqrt{\eps_n}$-neighborhood of each $x_i$; the limiting distribution at distinct points therefore decouples in the limit.

\begin{theo}
\label{thm: Plans}
Consider the setting of Theorem~\ref{thm: multivariate_limit} again and suppose that Assumption~\ref{assum: seq} holds.
Then, for any fixed $x\in \interior(\operatorname{supp}(\mu))$ and any sequence $\eps_n$ such that  $ n^{-2/(d+2)} \ll \eps_n \ll n^{-2/(d+6)}$ as $n \to \infty$,
\[
\sqrt{n}\, \eps_n^{d/4+1/2} \Big(  \hat{s}_{\eps_n}(x) -  \nabla \log \mu (x) \Big) \rightsquigarrow \mathcal{N}\big(0_d, \Sigma_d(x)\big),
\]
 for some covariance matrix $\Sigma_d(x) = C_d\mu(x)^{-1} I_d$, for some constant $C_d$ depending on the dimension only. 
\end{theo}
The proof is in Appendix~\ref{sec: Plans}.

Let us comment on the proofs of these two first main results. As the regularization parameter depends on the number of data points, the first key insight is to linearize by hand the difference between the population fixed point equation \eqref{eq:fixPtEq} and its empirical counterpart. A similar approach is used for the gradient of the potential. Then, once the linearization is shown to hold in a suitable regime, we argue with a Lindeberg--Feller central limit theorem. 
The main difficulties come from the necessity to first understand finely fixed point equations with random operators. 
Also, the limit reflects the fact that the sequence of functions considered in the Lindeberg--Feller central limit theorem is obtained from a Neumann series expansion.

\subsection{Two-measure case}
\label{sec: TwoMeas}

In this section, we will need the following assumption. 
\begin{assum}
\label{assum: Dim}
The dimension $d$ is strictly larger than 4.
\end{assum}
This assumption is necessary for the existence of the limiting variance in Theorem~\ref{thm: twoMeasures}.

Before developing further, let us introduce
$
K_{\eps_n}^\mu[h],
$
 the operator defined as 
\[
h\mapsto \int \pi_{\eps_n}(\cdot,y) h(y) \diff \mu(y).
\]
Note that the notation has been adapted to reflect the measure against which one integrates, which is necessary as the situation is not symmetric anymore. 
Such operators (for $\mu$ and $\nu$) are called the Sinkhorn operators. 
Similarly to the one-measure case, the first step of the proof is a linearization by hand, which we provide below.

\begin{prop}
\label{prop: Inversion}
Under the same setting as Assumption~\ref{assum: Main}, upon choosing $n^{-2/(d+2)} \ll \eps_n $,  it holds
\begin{align*}
&
\begin{pmatrix}
\hat f_{\eps_n} -f_{\eps_n}\\
\hat g_{\eps_n} -g_{\eps_n}
\end{pmatrix}\\
&= - \frac{\eps_n}{\sqrt{n}}\\
&\quad \times\begin{pmatrix}
 ( \id -  K_{\eps_n}^{{\nu}} K_{\eps_n}^{{\mu}})^{-1} 
& -K_{\eps_n}^{{\nu}} ( \id -  K_{\eps_n}^{{\mu}} K_{\eps_n}^{{\nu}})^{-1} \\
-( \id -  K_{\eps_n}^{{\mu}} K_{\eps_n}^{{\nu}})^{-1} K_{\eps_n}^{{\mu}} 
&  ( \id -  K_{\eps_n}^{{\mu}} K_{\eps_n}^{{\nu}})^{-1} 
\end{pmatrix}\begin{bmatrix}
 \int \pi_{\eps_n}(\cdot,y) \diff \mathbb{G}_n^{{\nu}}(y) \\
 \int \pi_{\eps_n}(x,\cdot) \diff \mathbb{G}_n^{{\mu}}(x) 
\end{bmatrix} \\
&+\oh_p( n^{-1/2} \eps_n^{-d/4 +1}).
\end{align*}
\end{prop}
The proof is presented in Section~\ref{sec: Inversion}. The fundamental work of \citet{gonzalez2022weak}  paved the way for such a result, with the important difference that the regularization parameter was then held fixed.  In the work just mentioned, the same type of operators were already appearing. We provide a clear understanding of what these operators look like for small regularization parameter. 
This is the content of the following theorem.

%--------------------------------- Main theorem --------------------

\begin{theo}
\label{thm: LimOp}
Under Assumption~\ref{assum: Main},  it holds that
\begin{align}
\label{eq: MeasLimOp}
 &   h\mapsto \nonumber K_{\eps_n}^{{\mu}}\Big[K_{\eps_n}^{{\nu}}[h]\Big](y_2)   \\\nonumber
   &  =   \frac{1 + \oh(1)}{\nu^{1/2}(y_2)} \int \frac{1 }{(2\pi\eps_n)^{d/2}} 
 \ h(y) \sqrt{ \frac{
    \det[ \nabla^2 \phi_0^{*}(y_2)]\det[  \nabla^2 \phi_0^{*}(y)]
}{
    \det [ \nabla^2 \phi_0^{*}(y_2)+\nabla^2 \phi_0^{*}(y)]
}
    } \\
&  \times \exp \left(-  \frac1{4\eps_n} 
(y - y_2)^\top \big[\nabla^2 \phi_0^{*}(y_2) \big](y - y_2) + \oh\left(\frac{\|y-y_2\|^2}{\eps_n}\right)  \right) 
 \nu^{1/2}(y) \diff y,
\end{align}
as $\eps_n \to 0$.
\end{theo}
The proof is in Appendix~\ref{sec: LimOp}.  This result is important. It shows how the combination of the two Sinkhorn operators appearing in Proposition~\ref{prop: Inversion} behave in the small epsilon limit. For a first intuition, recall that the identity minus the gaussian convolution properly rescaled converges to the Laplace operator\footnote{The reader can convince themself by a simple Taylor expansion.} (under suitable assumptions) as the variance decreases, i.e.,
\[
\frac{\id - k_\eps \star }{\eps} \to \Delta, \quad \text{ as } \eps \to 0. 
\]
Even though the result above is stated as a mere intuition the combination of Proposition~\ref{prop: Inversion} and Theorem~\ref{thm: LimOp} suggests a clear picture: 
\begin{center}
\emph{The fluctuations of the potentials are driven by a (sequence of) kernels solving (in $L^2$) a second-order elliptic differential equation where the data term is 
$e^{-D(x,y)/\eps }/ \big((2\pi\eps)^{d/2} \sqrt{\mu(x) \nu(y)} \big)$
at first order. 
This echoes with the fact that the linearization of the Monge--Ampère equation also is a second-order elliptic differential equation.}
\end{center}

\begin{prop}[The Sinkhorn fluctuation kernel]
\label{prop: SinkFluctKer}
 As $\eps_n \to 0$, the following asymptotic equivalence holds. 
\[
\left[( \id -  K_{\eps_n}^{{\nu}} K_{\eps_n}^{{\mu}})^{-1}  [\pi_{\eps_n}(x,y)]\right] \asymp \frac{\mathcal{F}^{-1}\left[ \tfrac{
\mathcal{F}\left[\exp\left(-\frac1{\eps_n} D(x,\cdot) \right)\right](\xi)
}{4 \pi^2 \xi^\top [\nabla^2 \phi(x) ]\xi }
\right](y)}{(2\pi\eps_n)^{d/2} \eps_n \nu^{1/2}(y)\mu^{1/2}(x)  }.
\]
\end{prop}

The proof is in Appendix~\ref{sec: SinkFluctKer}.

\begin{rmk}
The recent work of \cite{manole2023central} exploits a linearization of Monge--Ampère to provide central limit theorems for another estimator of the transport maps. We sketch their approach in Appendix~\ref{sec: Manole} for the readers to see the beautiful parallels. 
\end{rmk}

%--------------------------------- Main theorem --------------------

With this important fact understood, we can now turn to the central limit theorems in the two-measure setting. For simplicity, we only consider the two-measure transport problem where only $\mu$ is estimated by $\mu_n$ and $\nu$ is given.

\begin{theo}
\label{thm: twoMeasures}
Under Assumptions~\ref{assum: Main} and \ref{assum: Dim}, for fixed $x \in \operatorname{int}(K) $ and $n^{-2/d} \ll \eps_n $, 
    \[
    \sqrt{n}\eps_n^{d/4} \left(\hat f_{\eps_n}(x) -f_{\eps_n}(x)\right) \rightsquigarrow \mathcal{N}\big(0, \Sigma(x)\big), \quad \text{as }\eps_n\to 0, n \to \infty,
    \]
    with 
    \[
    \Sigma(x) = \lim_{n \to \infty}\frac1{\mu(x)} \int_{\reals^d} \left( \tfrac{ \mathcal{F}\left[ e^{-\frac1{\eps_n} D(x,\cdot)   }\right](\xi)}
{  (
2\pi\eps_n)^{d/4}\xi^\top \nabla^2 \phi(x) \xi }\right)^2 \diff \xi,
    \]
    where the last limit exists as $D(x, \cdot)$ has a quadratic behavior.
    
    Further, choosing $n^{-2/(d+2)}  \ll \eps_n \ll n^{-2/(d+4)}$, it holds that
    \[
    \sqrt{n}\eps_n^{d/4-1/2} \left(\nabla\hat f_{\eps_n}(x) - T_0(x)\right) \rightsquigarrow \mathcal{N}\big(0, \Sigma'(x)\big), \quad \text{as }\eps_n\to 0, n \to \infty,
    \] 
    with
    \[
    \Sigma'(x) = \lim_{n \to \infty}\frac1{\mu(x)} \int_{\reals^d}  \nabla^2 \phi(x) \xi \xi^\top \nabla^2 \phi(x)^\top \left( \tfrac{
 \mathcal{F}\left[ e^{-\frac1{\eps_n} D(x,\cdot)   }\right](\xi)}
{ \eps_n^{1/2} (2\pi\eps_n)^{d/4}  \xi^\top \nabla^2 \phi(x) \xi }\right)^2 \diff \xi,
    \]
\end{theo}

\begin{rmk}[Fully empirical version] If both marginals are estimated, the right-hand side of Proposition~\ref{prop: Inversion} becomes a sum of two independent empirical processes. Under the bandwidth assumptions of Theorem~\ref{thm: twoMeasures}, each contributes a diagonal term to the limiting covariance at distinct points, and the limit is the sum of two variance terms of the form given in Theorem~\ref{thm: twoMeasures}. 
\end{rmk}

\section{Application to score functions estimation}
\label{sec: AppliScore}

In this subsection, we exhibit rates of convergence in the $L^2(\mu)$ metric, which is the one of interest when it comes to studying the convergence diffusion schemes.
To continue, we require the following additional assumptions. By analogy with the classical notation used for log-concave measures, let us introduce the notation $V(x) := -\log \mu(x), x \in \operatorname{int}(K)$. 

\begin{assum}
   \label{ass: Reg'} 
   There exists a function $R(x) \in \cC^1(\operatorname{int}(K))$ such that,
    \begin{align*}
   & \Big\vert -2f_{\eps_n}(x)/{\eps_n}
     - \log\mu(x) - \frac{d}{2}\log(2\pi{\eps_n}) \\
  & \qquad  - {\eps_n} \operatorname{tr} \left( \frac{\nabla^2\mu(x)}{4\mu(x)} - \frac{1}8\big(\nabla \log \mu(x) \big)\big(\nabla \log \mu(x) \big)^\top \right)\Big\vert = {\eps_n}^2 R(x). 
    \end{align*}
\end{assum}

\begin{assum}
\label{ass: C}The Hessian of $V$ is Lipschitz, i.e.,  there exists $M > 0$ such that 
    $$\|\nabla^2 V(x) - \nabla^2 V(y)\| \leq M\|x-y\|.$$
     Furthermore, $\nabla^2  V \preceq \beta I_d$.
\end{assum}

\begin{theo}[$L^2(\mu)$ convergence rates]
\label{thm:main_stat_result} 
Under  Assumptions~\ref{assum: Main}, \ref{ass: Reg'} and  \ref{ass: C}. Further assume that the support is bounded. Setting the regularization parameter $\eps_n \asymp n^{-1/(d+4)}$, the estimator \eqref{eq: estimator_def} achieves the following rate of estimation
\begin{align}
    \E\|\hat{s}_{\eps_n} - \nabla \log \mu\|^2_{L^2(\mu)}\lesssim n^{-2/(d+4)}\,.
\end{align}
\end{theo}

\begin{rmk}
Our proofs, and thus our results, are asymptotic. Yet, 
they can be converted into finite sample bounds at the price of multiplying by sufficiently large (but universal) constants and taking $n$ larger than some $n_0 \in \NN $. 
\end{rmk}

The proof of Theorem~\ref{thm:main_stat_result} follows a traditional bias-variance trade-off arising from the inequality
\begin{align}
        \E\|\hat{s}_\eps - \nabla \log \mu\|^2_{L^2(\mu)}\leq 2\E\|\hat{s}_\eps - s_\eps\|^2_{L^2(\mu)} + 2 \|s_\eps - \nabla \log \mu\|^2_{L^2(\mu)}\,,
\end{align}
where $s_\eps$ is the population counterpart to our estimator.
The proof follows directly from the following two propositions. Proposition~\ref{prop: PopBias} controls the bias in $L^2$. We then control the stochastic fluctuations in Proposition~\ref{prop: StatFluct}, from which a trade-off in $\eps_n$ yields the final result; we omit this final step.

\begin{prop}[Population bias]
\label{prop: PopBias}
 Under Assumption~\ref{ass: Reg'} and \ref{ass: C},
    it holds that 
    \[
    \|s_\eps - \nabla \log \mu\|^2_{L^2(\mu)} \lesssim \eps^2 \big(M^2 d + \beta^2 I_0(\mu)\big),
    \]
where $I_0(\mu)$ is the Fisher information of $\mu$, and we otherwise omit universal constants.
\end{prop}

\begin{prop}[Statistical fluctuations]
\label{prop: StatFluct}
    Consider a sequence  $(\eps_n)_{n\geq1}$ with $\eps_n \to 0$ as $n\to\infty$.  Under Assumption~\ref{assum: Main},
    it holds that
    \[
    \E\|\hat{s}_{\eps_n} - s_{\eps_n}\|^2_{L^2(\mu)} \lesssim n^{-1} \eps_n^{-d-2}.
    \]  
\end{prop}

The two propositions above are proved in Appendix~\ref{sec: ProofScoreL2}. They build upon the linearization we had obtained.
Sections~\ref{sec: ConfirPot} and~\ref{sec: IntrDim}  illustrate Theorems~\ref{thm: multivariate_limit}  and~\ref{thm:main_stat_result}  numerically; Section~\ref{sec: IntrDim} further discussed the question of intrinsic-dimension adaptation raised at the start of this section.

\section{Numerics}
\label{sec:Numerics}

\subsection{On the computational aspects}

We now show that the computation of the self-transport potentials is easy as it is based on the fixed-point relation~\eqref{eq:fixPtEq}. Iterating this process on the basis of samples usually converges in a handful of iterations (typically in two to three iterations, as also observed by \citet{feydy2019interpolating}). We provide pseudo-code in Algorithm~\ref{alg:fixedpoint_alg} (``LSE'' refers to the ubiquitous log-sum-exp operator, which leads to well-conditioned updates).  

\begin{algorithm}[h!]
\caption{\label{alg:fixedpoint_alg} Computation of the self-OT potential}
\begin{algorithmic}
\State \textbf{Input:} $K_{\max} \geq 0$, data $(x_i)_{i=1}^n \sim \mu$, regularization parameter $\eps$
\State \textbf{Initialize:} $f^{(0)} \gets 0_n \in \R^n$ and $k \gets 1$ \Comment{$f_i^{(k)}\defeq f^{(k)}(x_i)$}
\While{$k \le K_{\max}$}
\State $f_i^{(k)} \gets \frac12 f_i^{(k-1)} - \frac\eps2 \LSE_{j=1}^n \Big[\log(\tfrac1n) + \tfrac1\eps f_j^{(k-1)} - \tfrac1{2\eps} \|x_i-x_j \|^2 \Big]$
\State $k \gets k+1$
\EndWhile
\State \textbf{Return:} $f^{(K_{\max})}$
\end{algorithmic}
\end{algorithm}

The two-sample problem has also been studied extensively. The classical Sinkhorn algorithm starts with a pair of potential candidates ($f_{\text{init}},g_{\text{init}})$ and updates recursively each potential for it to satisfy one of the conditions of equation \eqref{eq:dual_opt}.

Just like the algorithm in the pseudo-code above making ``half-Sinkhorn steps", a similar approach can be applied to the classical Sinkhorn algorithm.  This approach with ``partial updates" is the so-called overrelaxed version and is described in \cite{thibault2021overrelaxed}. The first steps are classical Sinkhorn steps and then, based on estimates, the overrelaxation parameter is used alike the 1/2 in the self-transport case. Importantly, note that the regularization parameter conditions that we obtained show that it should decrease very slowly with the sample size, implying that off-the-shelf algorithms might be safely used in most cases.

\subsection{Impact of tilting}
\label{sec: IntrDim}

We have showed that the gradient of the self-transport potentials can serve as estimators of the score function. The rate for estimating score functions is not matching the minimax one, as the current proof technique requires more regularity. Still, does it mean that the estimator needs to be forgotten? 
Understanding the impact of the tilting seen in equation \eqref{eq: estimator_def} is a natural objective.

In this experiment, we consider a Gaussian random variable that is supported on a 3-dimensional subspace in the case where the observed data is of dimension $d$, for $d\in\{5,10,50\}$. The $L^2$ error is estimated based on Monte-Carlo integration with 30000 points and the results are averaged over 5 replicates for more robustness of the results. The results are presented in Figure~\ref{1D: GaussSubsp}. We choose to present the result for two bandwidth size, the ambient one and the oracle one knowing the intrinsic dimension. The first choice corresponds to a case where the practitioner does not realize that the data is lower dimensional.  This choice further aims at verifying whether a sort of ``low-complexity adaptation principle" \citep{groppe2023lower,rigollet2025sample} could hold for the potentials.

From that graph, it is apparent that estimation via self-transport, using the bandwidth size that is optimal for the ambient dimension is performing better than the kernel density estimator with the same bandwidth choice in high dimensions.  In practice, for machine learning applications where the intrinsic dimension is unknown, this feature of self-transport might be an advantage of kernel density estimators. 

If one uses the unknown intrinsic dimension,  the slopes are similar and the constant chosen in front of the term $n^{-1/7}$ plays an important role. This choice of factor explains the underperformance for $d=5$ in the leftmost plot. 
The plot on the right hand side suggests that the estimator based on optimal transport is more variable. 

\begin{figure}[h!]
\centering
\includegraphics[width=0.47\textwidth]{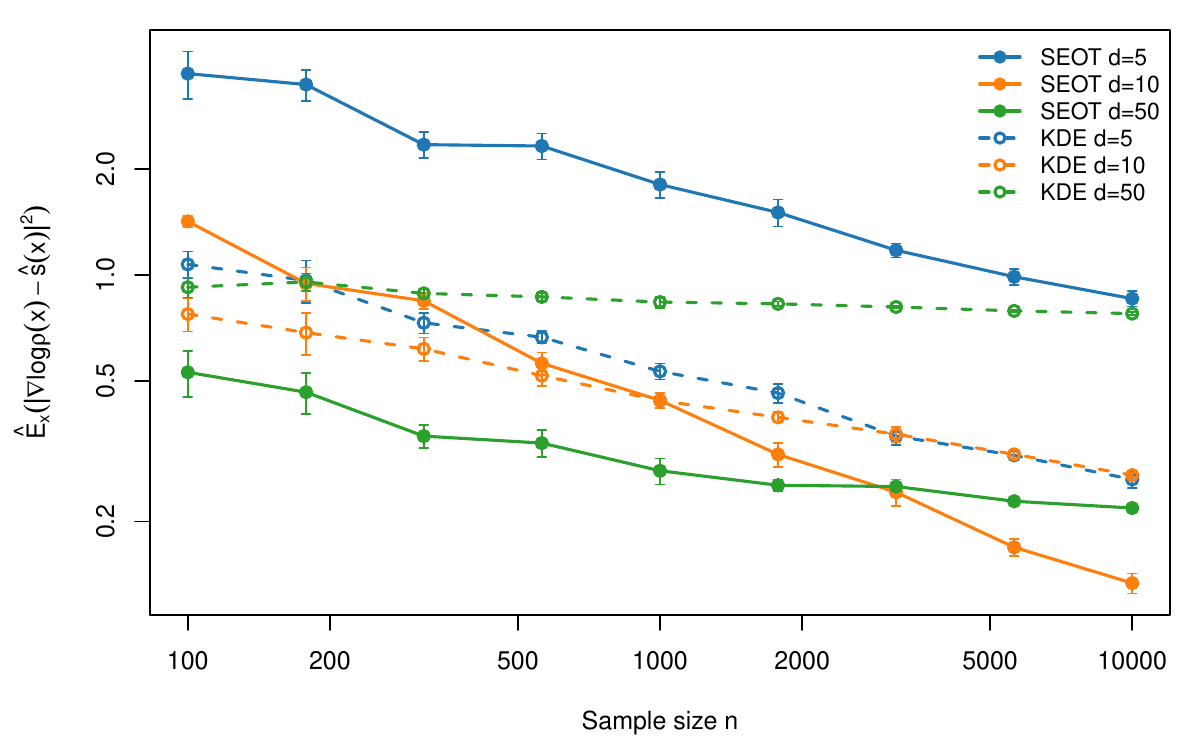}
\includegraphics[width=0.47\textwidth]{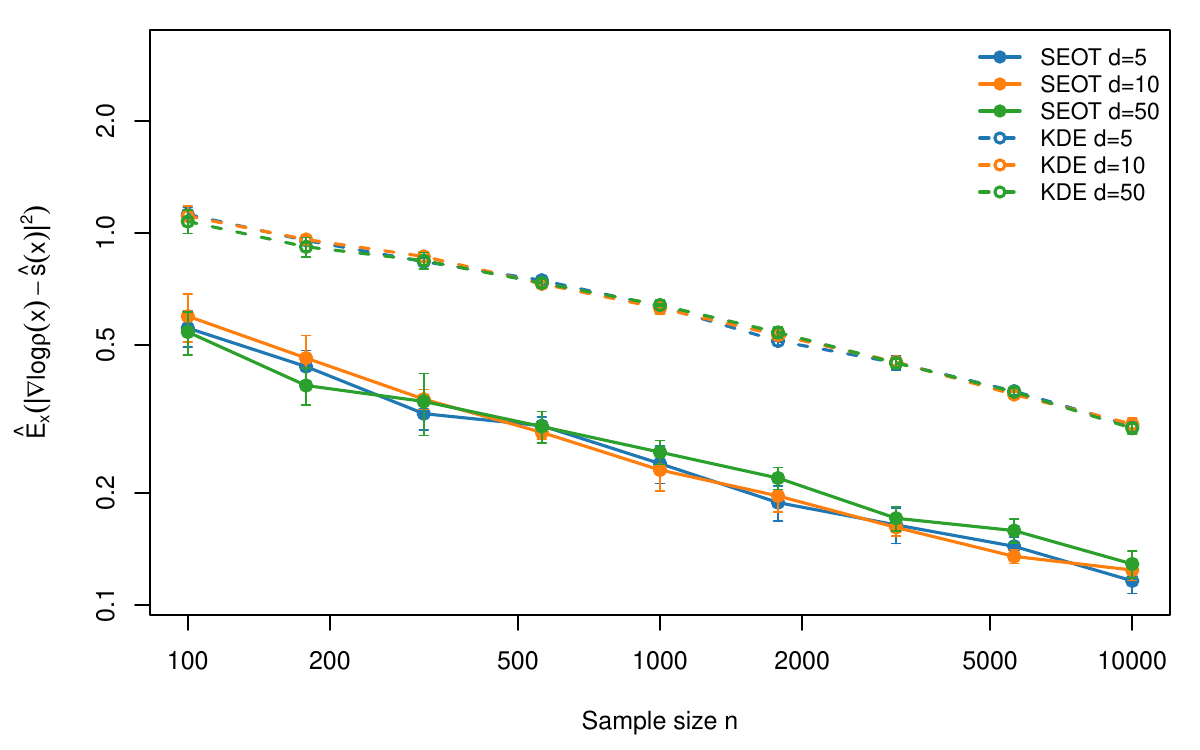}
\caption{\label{1D: GaussSubsp}  Estimation error (MC-approximated) of the self-ot scores and KDE in a Gaussian model for various dimensions $d\in\{5,10,50\}$. The covariance matrix is diagonal on the subspace on which it concentrates with diagonal (0.5, 1.5, 2.5).   
Left)  The na{\"i}ve equivalent bandwidth choice is $n^{-\frac{1}{d+4}}$. Right) The chosen bandwidth is  $2 n^{-\frac{1}{3+4}}$, incorporating the knowledge of the intrinsic dimension.
}
\end{figure}

In the particular case of the simulation, the formula of the fixed point equation might help understand the phenomenon at play. The measure $\mu$ is actually a push-forward $i_{\#}\gamma$, where $i$ adds coordinates and rotates the distribution. Calling $\reals^p$ the ambient space and $\reals^d$ the actual one,  we thus have
\begin{align*}
u_\eps(x) &= \int 
    \frac1{u_\eps(y)}\exp\left(-\tfrac1{2\eps} \|x-y\|^2\right)
    \diff i_{\#}\gamma(y) \\
    &= \int 
    \frac1{u_\eps( i(y))}\exp\left(-\tfrac1{2\eps} \|x- i(y)\|^2\right)
    \diff\gamma(y).
 \end{align*}
From which it is clear that $u_\eps(x)$ can capture the correct subset to reduce the problem to a smaller dimensional one, as $\|i(z)- i(y)\|_{\reals^p}^2= \|z- y\|_{\reals^d}^2$. 
In the case of data supported on a manifold, one would recover a similar behavior (at the first order) by working in local coordinates and observing that the kernel morally selects a small neighborhood.

\subsection{Confirmation of the CLT for the potentials}
\label{sec: ConfirPot}

In this final subsection, we numerically verify the limiting variance of 
Theorem~\ref{thm: multivariate_limit} for the self-transport potentials. We take 
$\mu = \mathcal{N}(0, I_d)$ and $x_0 = 0$, where the population potential 
$f_\eps$ is available in closed form (Appendix~\ref{sec: GausEOT}), 
so the rescaled statistic 
$T_b := \sqrt{n}\,\eps_n^{d/4-1}(\hat f_{\eps_n}(0) - f_{\eps_n}(0))$ 
can be centered exactly. The limiting variance reduces to $C_3(d)/4$.

\begin{figure}[h!]
  \centering
  \includegraphics[width=\textwidth]{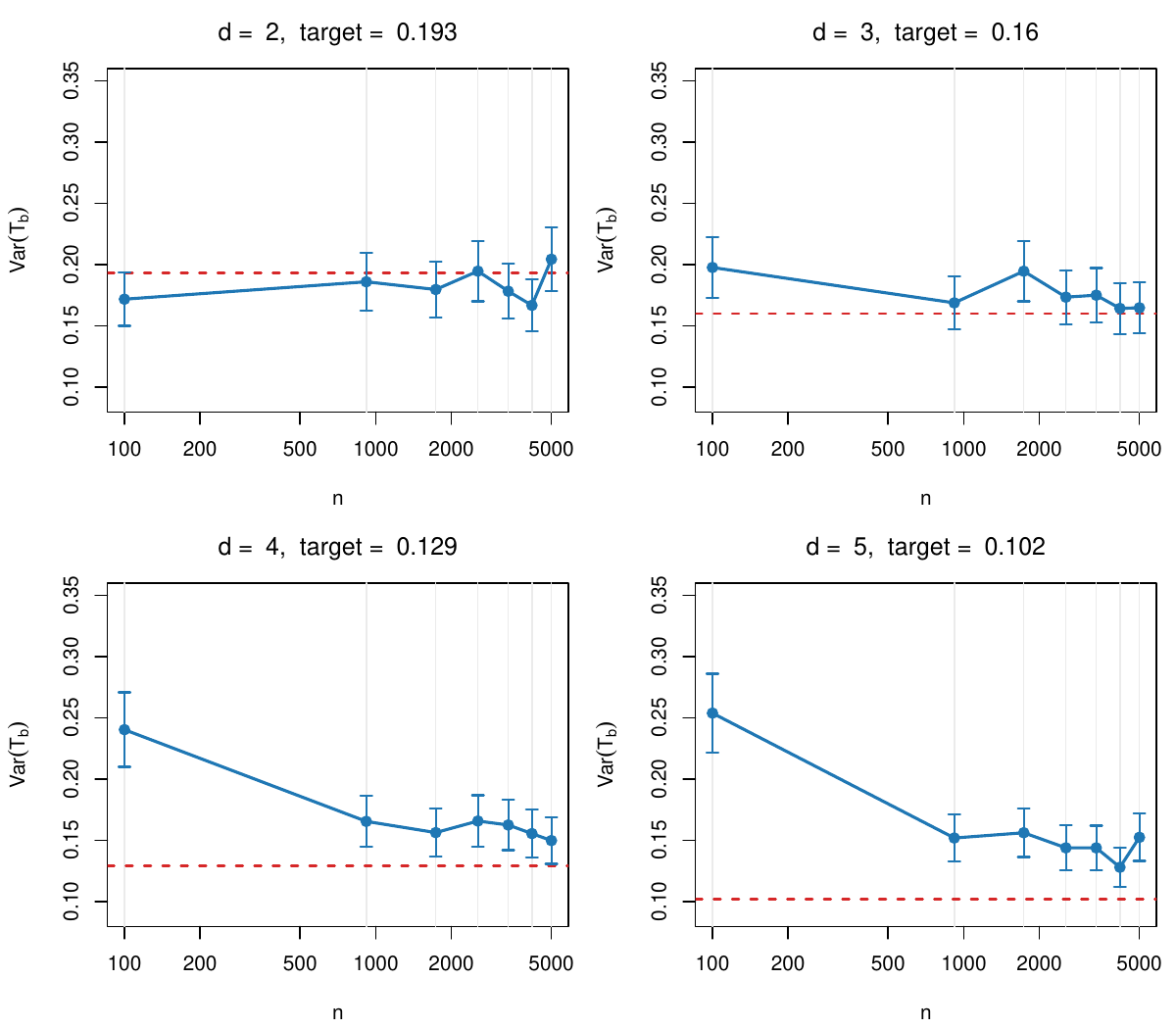}
  \caption{Empirical variance of $\sqrt{n}\,\eps_n^{d/4-1}(\hat f_{\eps_n}(0) - f_{\eps_n}(0))$ 
  vs.\ the theoretical limit $C_3(d)/4$ from Theorem~11, for $d \in \{2,3,4,5\}$ 
  and $\eps_n = n^{-2/(d+4)}$. Error bars show $\pm 2$ standard errors 
  ($B = 500$ replications). The dashed line is the asymptotic target.}
  \label{fig:clt}
\end{figure}

Figure~\ref{fig:clt} shows the empirical variance of $T_b$ for 
$d \in \{2,3,4,5\}$, $n$ up to $5000$, and $B = 500$ replications, choosing 
the parameter $\eps_n = n^{-2/(d+4)}$. For $d = 2$ and $d = 3$, 
the empirical variance matches $C_3(d)/4$ to within sampling uncertainty 
from $n = 1000$ onward. For $d = 4$ and $d = 5$, the ratio of the variance to $C_3(d)/4$ 
decreases monotonically toward~$1$ as $n$ grows, with the residual gap 
consistent with the $O(\eps_n\, d)$ correction from 
Proposition~\ref{prop:secOrdExp}. Reaching $\eps_n d \ll 1$ at 
$d = 5$ would require $n \approx 3 \times 10^5$.

%----------------------------------------------------------------------------------------------------
%                                     Conclusion 
%----------------------------------------------------------------------------------------------------

\section{Conclusion}
\label{sec: Conc}

In this paper, we have developed the first pointwise central limit theorems for transport maps estimated via the barycentric projection of entropy-regularized optimal transport. We have showcased the fact that entropy regularized self-transport behaves like a tilted and boundary-corrected kernel density estimator, culminating in the remarkable fact that Sinkhorn behaves in the small epsilon-limit just like the natural convolve-and-then-transport approach. 

The two approaches have pros and cons, which are important to put into perspective.
The Sinkhorn-based estimators are computationally easy but the implicit kernel is imposed to the user by the structure of the problem. The convolve-and-then-transport approach would enable to adapt into the regularity properties of the map if we knew them but is currently computationally out-of-reach. 
This raises the first open problem. 

\begin{opProb}
Can we come up with a regularization that still enables fast computations but allowing for the implicit kernel to adapt to smoothness? 
\end{opProb}

Ideas in that direction might involve quadratically regularized optimal transport, a variant that received a lot of attention recently, see \citet{zhang2023manifold,garriz2024infinitesimal,gonzalez2025sparse,gonzalez2025sample}; 
\citet{liu2025beyond} also recently proposed to explore new regularization schemes. 

Our work invites further developments. In particular, we conjecture that the tilt hinted at in Section~\ref{se: scoresDetour} might be useful for adaptation to intrinsic dimension, as hinted at by the simulations in Section~\ref{sec: IntrDim}. This raises the  following question.

\begin{opProb}
Building on the elements of the previous section, can one prove that the tilting induced by entropy-regularized optimal transport provides some adaptivity to the intrinsic dimension? Can one identify in what cases it brings an improvement? 
\end{opProb}

The proof certainly needs a nontrivial generalization of our results. It is expected that taking epsilon too large (even if decreasing) translates in an important bias, so that one expects that second problem to be mostly an analysis question.

Finally, even though our results are framed under natural assumptions arising from the optimal transport theory, they remain quite restrictive. This raises the next question.

\begin{opProb}
What happens to the estimators when the regularity conditions of the paper aren't met? Can we generalize to sub-Gaussian distributions? 
\end{opProb}

Beyond these specific problems, we hope the techniques of this paper, in particular the linearization of the empirical Schrödinger system at $\eps_n \to 0$ and the resolvent analysis of the Sinkhorn composition, will find application in other settings where entropic regularization is used as a statistical or computational tool.

%----------------------------------------------------------------------------------------------------
%                                     Acknowledgements
%----------------------------------------------------------------------------------------------------

\subsection*{Acknowledgements}
I am thankful to Ariane Fazeny, Alberto Gonz\'alez-Sanz, Shayan Hundrieser, Jonathan Niles-Weed, Aram-Alexandre Pooladian, Tudor Manole and Thomas Staudt for helpful discussions and encouragements. 
The support of the Deutsche Forschungsgemeinschaft through SFB1456 is further acknowledged.

 \bibliography{ref}

\newpage
\appendix

%----------------------------------------------------------------------------------------------------------------------------------------------
%----------------------------------------------------------------------------------------------------------------------------------------------
%
%                                                       APPENDICES
%
%----------------------------------------------------------------------------------------------------------------------------------------------%----------------------------------------------------------------------------------------------------------------------------------------------

\section{Additional results}

\subsection{EOT in the Gaussian setting}
\label{sec: GausEOT}
In this section we first state the results about the potentials in the Gaussian case. 

This is the content of Theorem~ 2 in \citet{mallasto2021entropy}, which we slightly rephrased and took out the part of interest to us.

\begin{theo}
Let $\mu_i = \mathcal{N}(0,\Sigma_i)$ for $i\in {0,1}$ be two centred Gaussian distributions in $\reals^d$, write 
$N_{i,j}^\eps = \left( I + \tfrac{16}{\eps^2}  \Sigma_i^{1/2} \Sigma_j \Sigma_i^{1/2}\right)^{1/2}$ and $M^\eps = I + \left( I + \tfrac{16}{\eps^2}  \Sigma_0 \Sigma_1\right)^{1/2}.
$
Then, 
the density of the optimal transport plan is given by 
\[
\pi^\eps(x,y) = \alpha^{\eps}(x) \alpha^{\eps}(x) \exp\left(- \frac{\|x-y\|^2}{\eps} \right) \mu_0(x) \mu_1(y),
\]
where 
\(
\alpha^\eps(x) = \exp( x^\top A x +a ),
\)\(
\beta^\eps(y) = \exp( y^\top B y +b )\) and 
\begin{align*}
&A = \frac14 \Sigma_0^{-1/2} \left( I + \frac4\eps \Sigma_0 - N_{01}^\eps \right)  \Sigma_0^{-1/2}\\
&B = \frac14 \Sigma_1^{-1/2} \left( I + \frac4\eps \Sigma_1 - N_{10}^\eps \right)  \Sigma_1^{-1/2}\\
&\exp(a+b) = \sqrt{2^{-n} \det M^\eps}.
\end{align*}
\end{theo}

\subsection{The approach of Manole et al. to estimate the transport map.}
\label{sec: Manole}

In the recent work by \cite{manole2024plugin}, the authors propose to rely on kernel density estimators the Monge--Ampère equation to estimate transport plans given, as above, an i.i.d.\ sample. In this appendix we will just explain the ideas so that the reader can compare their results to the picture developed in  Section~\ref{sec: TwoMeas}. 

Start from the Monge--Ampère equation, assuming two probability densities $p,q$ nonvanishing on the torus,
\[
\det \nabla^2 \phi_0 = \frac{p}{g \circ \nabla \phi_0}.
\]
Then, if one replaces $q$ by a kernel density estimator $\hat q_h$ with bandwidth $h$ in the equation above, one has 
\[
\det \nabla^2 \hat  \phi = \frac{ p}{\hat g_h \circ \nabla \hat\phi}.
\]
As $\hat p_h$ converges to $p$ as the sample size increases and the bandwidth decreases, 
\cite{manole2024plugin} prove that 
\[
 \hat  \phi - \phi_0 \asymp L^{-1}[ \hat q_h - a], 
\]
where 
\[
L[u] = - \operatorname{div} \big( q \nabla u (\nabla \phi_0^*)\big),
\]
where $\operatorname{div}$ is the divergence operator and the asymptotic equivalence is in a suitable topology.
Their theorem 5 then shows that, upon choosing properly the bandwidth $h_n$ as a function of the sample size and for sufficiently smooth densities on the torus,
\[
\sqrt{n h_n ^{d-2}} \left(\hat T_n (x) - T_0(x)\right) \dto \mathcal{N}(0, \Sigma(x)), \quad \text{as } n \to \infty,
\]
where 
\[
\Sigma (x) = \frac1 {p(x)} \int_{ \reals^d} \xi\xi^\top \left( 
\frac { \mathcal{F}[\cK](\mathcal{M}(x)\xi)}
{
2\pi \langle \mathcal{M}(x)\xi, \xi\rangle) 
}
\right) \diff \xi, \quad  \text{with}\ \mathcal{M}(x) := \nabla^2 \phi^*_0( \nabla^2 \phi_0(x)), 
\]
where $\cK$ is the kernel used.

\subsection{Fourier intuition}
\label{sec: AddRes}

Up to shifting the cost function by $ \eps d \log (2\pi\eps) $, we can rewrite \eqref{eq:fixPtEq} as 
\[
u_\eps = \gamma_\eps * \left(\frac{\mu}{u}\right),
\]
with $\gamma_\eps$ the isotropic centred Gaussian density with variance $\eps$. 
Assuming that the Fourier transform of both sides is well defined, expanding in $\eps$ the Fourier transform of $\gamma_\eps$, we get 
\[
\mathcal{F}(u_\eps)[\xi] = \mathcal{F}\left(\frac{\mu}{u_\eps}\right)[\xi]\left( 1 -  \frac\eps{2}\|\xi\|^2 + \oh(\eps)\right).
\]
Setting $u_\eps = u_0 + \eps u_1 + \oh(\eps)$, a further expansion gives
\[
\mathcal{F}(u_0)[\xi] + \eps \mathcal{F}(u_1)[\xi] =  \mathcal{F}\left(\frac{\mu}{u_0}\right)[\xi]- \eps \mathcal{F}\left(\frac{\mu u_1}{u_0^2}\right)[\xi] - \frac\eps{2}\|\xi\|^2\mathcal{F}\left(\frac{\mu}{u_0}\right)[\xi] + \oh(\eps).
\]
 At the constant scale, $\mathcal{F}(u_0)[\xi] = \mathcal{F}\left({\mu}/{u_0}\right)[\xi]$ enforces $u_0=\sqrt{\mu}$. At the $\eps$ scale, one thus gets the equation
 \[
 \mathcal{F}(u_1)[\xi] = -\frac14\|\xi\|^2\mathcal{F}\left(\sqrt{\mu}\right)[\xi], 
 \]
 which is characteristic for elliptic differential equations, as expected from the trace term in Proposition~\ref{prop:secOrdExp}.
 Note that this view on the problem only requires the Fourier transform of the  various quantities of interest to be defined, which is quite mild.

\subsection{The constant of Theorem~\ref{thm: multivariate_limit}}
\label{sec: RomanNumbers}
Note that $C_3$ is related to binomial transform and Roman harmonic numbers \citet{sesma2017roman}. We deem the full simplification of this quantity beyond the scope of the current paper. However, a first simplification for $d$ even is the content of the following remark.
\begin{rmk}
For $d/2 \in \NN$, using the equality \citep[Lemma~2.1]{kim2023some}
\[
\int_0^1 (1-t)^{\kappa'} (\log t)^{d/2} t^{\eta+1 }\diff t = (d/2-1)! \sum_{\eta'=0}^{\kappa'} {\kappa'\choose\eta'} (-1)^{\eta' - (d/2-1)},
\]
we can 
simplify 
\begin{multline*}
\sum_{\eta=0}^\kappa\sum_{\eta'=0}^{\kappa'} {\kappa \choose  \eta} ( -1)^\eta ( -1)^{\eta'} {\kappa' \choose  \eta'} 
( \eta + \eta' +2)^{-d/2} \\
 = \frac{(-1)^{1-d/2}}{(d/2 -1)!} \int_0^1 (1-t)^{\kappa+\kappa'} t (\log t)^{d/2-1} \diff t,
\end{multline*}
for $\kappa, \kappa'>0$.
\end{rmk}

\section{Proofs of Propositions}
\label{app:Proofs}

\subsection{Proof of Proposition~\ref{prop:secOrdExp}}
\label{sec:secOrdExp}

\begin{proof}[Proof of Proposition~\ref{prop:secOrdExp}]
It holds that 
\begin{align*}
    u(x)
    &= \int \frac{k_\eps(x,y)}{u(y)} \diff {\mu}(y) \\
    &= \eps^{d/2} \int \frac{\mu(x+ \sqrt{\eps} z)}{u(x+ \sqrt{\eps} z)}   \phi(z)\diff z.
  %  &= \frac{\mu(x)}{u(x)} \eps^{d/2} \int \frac{1+ \sqrt{\eps} \nabla \log \mu(x)^\top z+ \eps z^\top \nabla^2{\mu(x)}z /(2\mu(x))+\oh(\eps)}{1+ \sqrt{\eps}\nabla \log u(x)^\top z + \eps z^\top \nabla^2{u(x)}z/(2u(x))+ \oh(\eps)}   \phi(z)\diff z,
\end{align*}
where $\phi(z)= \exp(- \|z\|^2/2)$. 
Then, using the assumptions of the sequence of potentials, 
\begin{align*}
 &u(x)\\
 &= \frac{\mu(x)}{u(x)} \eps^{d/2} \int \frac{1+ \sqrt{\eps} \nabla \log \mu(x)^\top z+ \eps z^\top \nabla^2{\mu(x)}z /(2\mu(x))+\oh(\eps)}{1+ \sqrt{\eps}\nabla \log u(x)^\top z + \eps z^\top \nabla^2{u(x)}z/(2u(x))+ \oh(\eps)}   \phi(z)\diff z .
\end{align*}

Let us focus on the first summand in the display above and develop the inner fraction to get
\begin{align*}
&\frac{1+ \sqrt{\eps} \nabla \log \mu(x)^\top z+ \eps z^\top \nabla^2{\mu(x)}z /(2\mu(x)) +\oh(\eps)}{1+ \sqrt{\eps}\nabla \log u(x)^\top z+\eps z^\top \nabla^2{u(x)}z /(2u(x)) + \oh(\eps)} \\
&\qquad \qquad\qquad \qquad=\big( 1+ \sqrt{\eps} \nabla \log \mu(x)^\top z+ \eps z^\top \nabla^2{\mu(x)}z/(2\mu(x)) +\oh(\eps)\big)\\
&\qquad \qquad\qquad \qquad\times\Big(1- \sqrt{\eps}\nabla \log u(x)^\top z+ \eps z^\top \nabla^2{u(x)}z /(2u(x)) \\
 &\qquad \qquad \qquad\qquad \qquad+ \Big(\sqrt{\eps}\nabla \log u(x)^\top z\Big)^2 + \Oh \left( \eps^{3/2} \right) +  \oh(\eps)\Big).
\end{align*}
Above, the term of order $ \Oh \left( \eps^{3/2} \right)$ behaves like 
\[
 \eps^{3/2} \nabla \log \mu(x)^\top z  \frac{z^\top \nabla^2{\mu(x)}z}{2\mu(x)} + C \eps^{3/2}  \|z \|^2,
\]
owing to the relationship between $u$ and $\mu$, which we show below. It is thus negligible once integrated against a standard Gaussian.
Also, the integral of the second-order Taylor residual of $\mu$ with respect to the Gaussian measure is well-controlled; it
has the form 
\[
\sum_{\vert \alpha \vert =2} \int h_\alpha( x +  \sqrt{\eps} z) ( \sqrt{\eps } z)^\alpha e^{-\frac12 \|z\|^2} \dd z,
\]
for functions $h_\alpha$ such that $\lim_{z\to x} h_\alpha( z)=0$. Clearly, 
\[
\eps \int h_\alpha( x +  \sqrt{\eps} z) z^\alpha e^{-\frac12 \|z\|^2} \dd z =\oh(\eps).
\]
Altogether,
\begin{align*}
&u(x)^2= \mu(x) \eps^{d/2} \\
& \quad \times\int  \Big( 1 + \eps z^\top \Big( \frac{\nabla^2\mu(x)}{2\mu(x)} - \frac{\nabla^2u(x)}{2u(x)} - \nabla \log \mu(x) \nabla\log u(x)^\top\\ & \qquad\qquad +  \nabla \log u(x) \nabla\log u(x)^\top \Big)z + \oh(\eps)\Big)  \phi(z)\diff z
\end{align*}
Finally, 
\begin{align*}
&u(x)^2
= \mu(x) (2\pi\eps)^{d/2}  \left( 1 + \eps \operatorname{tr} \left( \frac{\nabla^2\mu(x)}{2\mu(x)} - \frac{\nabla^2u(x)}{2u(x)}  \right)+ \oh(\eps)\right) . 
\end{align*}
We can further refine this result by ``bootstrapping'', i.e., observing that the asymptotic expansion $\nabla \log u(x) =  \nabla \log \mu^{1/2} +\oh(1)$ can be plugged-in. Therefore, 
\[
\nabla \log \mu(x) \nabla\log u(x)^\top = \frac12\nabla \log \mu(x) \big(\nabla \log \mu(x)  \big)^\top +\oh(1).
\]
Also, 
\begin{align*}
   & \nabla^2u(x)\\
    &= (2\pi\eps)^{d/4} \mu^{1/2}(x) \big(1+\oh(1)\big)\\
    &=(2\pi\eps)^{d/4}  \nabla \left(\frac12 \mu^{-1/2}(x)\nabla \mu(x)\right)\big(1+\oh(1)\big) \\
    &=(2\pi\eps)^{d/4}  \frac12\left(\frac{-1}2 \mu^{-3/2}(x) \nabla \mu(x) \big(\nabla \mu(x)\big)^\top
+\mu^{-1/2}(x) \nabla^2\mu(x)
\right)\big(1+\oh(1)\big) 
\end{align*}
Thus, 
\begin{align*}
    &\operatorname{tr} \left( \frac{\nabla^2\mu(x)}{2\mu(x)} - \frac{\nabla^2u(x)}{2u(x)} - \nabla \log \mu(x) \nabla\log u(x)^\top \right) \\
    &\qquad = 
    \operatorname{tr} \Big( \frac{\nabla^2\mu(x)}{2\mu(x)} + \frac18\nabla \log \mu(x) \big(\nabla \log \mu(x) \big)^\top
    - \frac{\nabla^2\mu(x)}{4\mu(x)} \\
   & \quad\qquad\qquad - \frac14\nabla \log \mu(x) \big(\nabla \log \mu(x)  \Big)^\top \Big) +\oh(1)
\end{align*}
and the claim follows.
\end{proof}

\subsection{ Proof of Proposition~\ref{prop: PotSingMeasL2}}
\label{sec: PotSingMeasL2}

\begin{proof}{Proof of Proposition~\ref{prop: PotSingMeasL2}}

We start with the first part of the claim.
First note that, for $f$ a continuous function on a compact set $U$, 
\[
\lim_{\eps \to 0}\gamma_\eps  * f - f   =0
\]
almost everywhere on any compact included in $U$.
Furthermore, 
\[
\lim_{\eps \to 0}  \|\gamma_\eps  * f - f \|_{L^2(\mu)}  =0.
\]

Consider the operator 
\[
\psi \mapsto \psi(x) + \log \int \gamma_\eps ( x, y)  e^{\psi(y)} \mu (y) \diff  y. 
\] 
We aim at evaluating it at two quantities, $\psi^* := f_\eps/\eps$ and $\psi =-\log( \mu )/ 2$. 
Set 
\[
E_\eps(x) =  -\log( \mu(x) )/ 2+ \log \int \gamma_\eps ( x, y)   \sqrt{\mu (y)} \diff  y.
\]
Set $\psi_t = t \psi + (1-t) \psi^* $.
Then, 
\begin{align*}
E_\eps(x) &= \big(\psi(x) - \psi^*(x)\big) + \int_0^1 \frac{\int \gamma_\eps ( x, y)  e^{\psi_t(y)} \big(\psi(y) - \psi^*(y)\big) \mu (y) \diff  y  }{  \int \gamma_\eps ( x, z)  e^{\psi_t(z)} \mu (z) \diff  z} \diff t \\
&=: (\id + K_{\bar t, \eps} )[ \psi - \psi^*](x),
\end{align*}
by the mean value-theorem.
Observe that $K_{\bar t, \eps}$ is a bounded linear operator from $\mathcal{C} \to \mathcal{C}$ and nonexpansive.  One also sees that this operator admits the eigenvalue 1 corresponding to the eigenfunction 1.   This operator can be seen as the dual Markov operator to a Markov chain and $(\id + K_{\bar t, \eps} )^{-1}$ exists. 
Further remark that this operator interpolates between 
\begin{equation}
\label{eq: SOTop}
f \mapsto P[f]:= \int \pi_\eps(x,y) f(y) \mu(y) \diff (y) = \mathbb{E}_{\pi_\eps}[f(Y) \vert X=x].
\end{equation}
where the remainder is uniform over compact subsets. 

The operator \eqref{eq: SOTop} is self adjoint in $L^2(\mu)$ and owing to symmetry and the positivity of the Gaussian kernel, it has a spectrum in $[0,1]$. 

Furthermore, the spectrum of the operator 
\[ 
f \mapsto \frac{\int \gamma_\eps ( x, y) f(y) \sqrt{\mu (y)} \diff  y  }{  \int \gamma_\eps ( x, z)   \sqrt{\mu (z)} \diff  z} %= \frac{ f(x) \sqrt{\mu (x)} + \oh(1) }{ \sqrt{\mu (x)} +\oh(1) }  \quad \text{as } \eps \to 0,
\] 
is positive as well as one sees by transforming the eigenvalue equation 
\[
\frac{\int \gamma_\eps ( x, y) f(y) \sqrt{\mu (y)} \diff  y  }{  \int \gamma_\eps ( x, z)   \sqrt{\mu (z)} \diff  z} = \lambda f(x),
\]
via the change  of variable $\tilde f(x) = \mu^{1/4}(x)\sqrt{ \int \gamma_\eps ( x, z)   \sqrt{\mu (z)} \diff  z }$, into the symmetric version 
\[
\int \frac{\mu^{1/4}(x) \gamma_\eps ( x, y) \tilde f(y)\mu^{1/4}(y)\diff  y  }{ \sqrt{ \int \gamma_\eps ( x, z)   \sqrt{\mu (z)} \diff  z}\sqrt{ \int \gamma_\eps ( y, z)   \sqrt{\mu (z)} \diff  z}} = \lambda  \tilde f(x).
\]
It follows that $(\id + K_{\bar t, \eps} )^{-1}$ has a bounded spectral norm that is uniform in $\eps$.
\end{proof}

%----------------------------------------------------------------------------------
%     Proposition 
%--------------------------------------------------------------------------------

\subsection{Proof of Proposition~\ref{prop: TwoMeasFixPt}}
\label{sec: TwoMeasFixPt}

\begin{proof}[Proof of Proposition~\ref{prop: TwoMeasFixPt}]

Combining the Schrödinger system with the definition of the residual in the Léger--Vialard expansion, we get
\begin{equation}
\label{eq: SingSystTwoMeas}
\frac{\tilde f_\eps(x)}{\eps} 
= - \log \left( \int  \frac{e^{-D(x,y)/\eps} \sqrt{\frac{\nu(y) }{\mu(x)}}}{
 \int e^{-D(z,y)/\eps}\ \sqrt{\frac{\mu(z) }{\nu(y)}}\ e^{\tilde f_\eps(z)/ \eps } \diff z
 } \diff y  \right).
\end{equation}

Further, appealing to the local quadraticity of the OT-divergence,
\begin{multline}
\label{eq: OTGausKern}
e^{-D(x,y)/\eps}\sqrt{ \frac{{\nu}(y)}{ \mu(x)}}  \\
=\exp \left( \frac{ -(y-x^*)^\top\nabla^2\phi_0^*( x^*)(y-x^*)+\oh(\|y-x^*\|^2) }{2\eps}  \right) \\
\times \frac{{\nu}^{1/2}(y) \sqrt{\det [\nabla^2 \varphi_0^*(x^*)]}}{ \sqrt{\nu(x^*)}}  
\end{multline}
and similarly for the other integrand. 

At the first order, the kernel in the display above is an approximation of identity. 
We thus see that setting $e^{\tilde f_\eps(z)/ \eps } =1$ ensures that both sides of equation~\eqref{eq: OTGausKern} are equal up to terms of order $\eps$ in the interior of the support.

A linearization similar to that carried out in Proposition~\ref{prop:secOrdExp}, yields a system of the form 
\[
\frac{\tilde f_\eps}{\eps} =
 \left(
 	 \id - \frac{K_{\eps}^2}{ K_{\eps}[1]} 
  \right)^{-1}
 \left[ 
  \log 
  \left( 
  \int  \frac{        e^{-D(x,y)/\eps} \sqrt{       \frac{   \nu(y) }  {\mu(x)}      }
  }{
 \int e^{-D(z,y)/\eps}\ \sqrt{     \frac    {\mu(z) }{\nu(y)}             } \diff z
 } 
 \diff y  
 \right) \right].
\]
To converge to a differential operator in $L^2$, the inverse in the display above needs to be rescaled by some power of $\eps$, while the logarithm converges to 0. The claim follows.  
\end{proof}

\subsection{Proof of Proposition~\ref{prop: otBregmanExp}}
\label{sec:otBregmanExp}

\begin{proof}[Proof of Proposition~\ref{prop: otBregmanExp}]
The strategy is the Laplace method applied to each of the two fixed-point equations, pushed one order beyond the leading Monge-Ampère relation. At leading order in $\eps$, the identity reduces to $\mu(x_0) = \nu(y)\det\nabla^2\varphi(x_0)$, which is precisely the Monge-Ampère equation. The $\Oh(\eps)$ correction comes from four contributions: the curvature of $\sqrt{\mu}$, a cross-term between $\nabla\sqrt{\mu}$ and the cubic in $v$ from the Taylor expansion of $D$ in direction $v$, the quartic term, and the square of the cubic, which we compute relying on Wick's theorem.

Fix $y$ and apply the Laplace method to the second equation of the system at the unique minimum $x_0 = T^{-1}(y)$ of $D(\cdot, y)$, where $\nabla^{2}_{xx} D(x_0, y) = A(y)^{-1}$. Make the change of variables $x = x_0 + \sqrt{\eps} v$ and Taylor-expand, using $D(x_0, y) = 0$, $\nabla_x D(x_0, y) = 0$, and $\partial^{k}_{x} D(x_0, y) = \partial^{k} \phi(x_0)$ for $k \ge 2$. This gives
\[
  -\tfrac{D(x_0 + \sqrt{\eps} v, y)}{\eps}
  = -\tfrac{1}{2} v^{\top} A^{-1} v
  - \tfrac{\sqrt{\eps}}{6} \phi_{ijk} v_i v_j v_k
  - \tfrac{\eps}{24} \phi_{ijk\ell} v_i v_j v_k v_\ell + O(\eps^{3/2}),
\]
and $\sqrt{\mu(x_0 + \sqrt{\eps} v)} = \sqrt{\mu(x_0)} + \sqrt{\eps} \nabla\!\sqrt{\mu} \cdot v + \tfrac{\eps}{2} v^{\top} \nabla^{2}\!\sqrt{\mu}\, v + O(\eps^{3/2})$. 

At leading order the identity reduces to $\sqrt{\mu(x_0) \det A(y) / \nu(y)} = 1$, which is the Monge--Amp\`ere equation $\mu(x_0) = \nu(y) \det \nabla^{2} \phi(x_0)$.

At order $\eps$, four contributions appear, evaluated via Wick's theorem (odd moments vanish, the quartic moment gives $A_{ij}A_{k\ell} + A_{ik}A_{j\ell} + A_{i\ell}A_{jk}$, the sextic moment gives the sum over the $15$ pairings).
First,  the Hessian of $\sqrt{\mu}$, producing the first trace term;
then, the cross-product of $\nabla\!\sqrt{\mu} \cdot v$ with the cubic in $v$, producing the $\phi_{ijk} \partial_\ell \sqrt{\mu}$ term.
  Finally,  the quartic $\phi_{ijk\ell}$ term, as well as, 
  the square $\tfrac{1}{2}(\tfrac{1}{6} \phi_{ijk} v_i v_j v_k)^{2} = \tfrac{1}{72} \phi_{ijk} \phi_{i'j'k'} v_i \cdots v_{k'}$, whose Gaussian average yields the sum over $15$ pairings.

Taking the log, dividing by the leading factor, and moving everything to the left gives the stated expansion, after splitting
\[
  -\tfrac{1}{4} \operatorname{tr}(A \nabla^{2} \log \mu)
  = -\tfrac{1}{2 \sqrt{\mu}} \operatorname{tr}(A \nabla^{2}\!\sqrt{\mu})
    + \tfrac{1}{2 \mu} \nabla\!\sqrt{\mu}^{\top} A \nabla\!\sqrt{\mu}
\]
into the two displayed terms. The first equation is treated identically with $(\mu, \phi, A)$ replaced by $(\nu, \phi^{*}, [\nabla^{2} \phi^{*}(T(x))]^{-1})$, yielding the symmetric identity.
\end{proof}

%----------------------------------------------------------------------------------
%     Proposition 11
%--------------------------------------------------------------------------------

\subsection{Proof of Proposition~\ref{prop: Inversion}}
\label{sec: Inversion}

\begin{proof}[Proof of Proposition~\ref{prop: Inversion}]

The proof strategy is to first control the fluctuations of the empirical potentials vis-à-vis their population counterparts in a sufficiently strong topology.
Then, the linearization will follow from the exact same argument as in the proof of Theorem~\ref{thm: multivariate_limit}.
Finally, we'll have to invert the operators appearing in the linearization to get to the claim.

\paragraph{Control of fluctuations}
Recall the definition of $\tilde f_{\eps_n}, \tilde g_{\eps_n}$ from
\eqref{eq: ExpTwoMeas}. We focus on their empirical counterpart, that is,
the pair of functions\footnote{The pair is unique up to additive constants.}
$(\bar f_{\eps_n}, \bar g_{\eps_n})$ solving
\begin{align}
\label{eq: defBar}
0 &= \bar f_{\eps_n}(x) + \eps_n \log \int
   \frac{e^{(\bar g_{\eps_n}(y) - D(x,y))/\eps_n}}
        {(2\pi\eps_n)^{d/2}\sqrt{\nu(y)\mu(x)}}\,\diff\nu_n(y), \\
        \nonumber
0 &= \bar g_{\eps_n}(y) + \eps_n \log \int
   \frac{e^{(\bar f_{\eps_n}(x) - D(x,y))/\eps_n}}
        {(2\pi\eps_n)^{d/2}\sqrt{\nu(y)\mu(x)}}\,\diff\mu_n(x).
\end{align}

The first objective is to ensure that 
\[
\|\bar f_{\eps_n} -\tilde f_{\eps_n} \|_\infty/\eps_n , \|\bar g_{\eps_n} -\tilde g_{\eps_n} \|_\infty/\eps_n  \to 0, 
\]
as $\eps_n \to 0$.  Further note that since $(\bar f_{\eps_n}, \bar g_{\eps_n})$ and
$(\hat f_{\eps_n}, \hat g_{\eps_n})$ differ only by deterministic shifts
($f_0$, $g_0$, $\tfrac{\eps_n}{2}\log\mu$, $\tfrac{\eps_n}{2}\log\nu$,
$\tfrac{\eps_n d}{4}\log(2\pi\eps_n)$), the same identity holds for
$\hat f_{\eps_n} - f_{\eps_n}$ and $\hat g_{\eps_n} - g_{\eps_n}$, which is
the statement required for the linearization that we aim at invoking.

We denote the above system compactly by
\[
F_{\eps_n}\!\begin{pmatrix}\bar f_{\eps_n}\\ \bar g_{\eps_n}\end{pmatrix}
= 0.
\]

A second-order Taylor expansion of $F_{\eps_n}$ at $(\tilde f_{\eps_n},
\tilde g_{\eps_n})$ yields
\begin{align*}
0 = &F_{\eps_n} 
\begin{pmatrix}\tilde f_{\eps_n}\\ \tilde g_{\eps_n}\end{pmatrix}
+ 
	\begin{pmatrix} 
		I & K_{\eps_n,n}^{\nu_n} \\
		K_{\eps_n,n}^{\mu_n} & I 
	\end{pmatrix}
\begin{pmatrix}\bar f_{\eps_n} - \tilde f_{\eps_n}\\
                \bar g_{\eps_n} - \tilde g_{\eps_n}\end{pmatrix} 
                \\
&+ \tfrac{1}{2\eps_n}
\begin{pmatrix} 
0 &  \mathrm{Var}_{\pi_{x,\eps_n,n}}^{\nu_n}\\ 
\mathrm{Var}_{\pi_{x,\eps_n,n}}^{\mu_n} & 0
\end{pmatrix}
\begin{pmatrix}\bar f_{\eps_n} - \tilde f_{\eps_n}\\
                \bar g_{\eps_n} - \tilde g_{\eps_n}\end{pmatrix},
\end{align*}
where the operators and the variances are defined as follows.

For each fixed $x$, let $\pi_{x,\eps_n,n}^{\nu_n}$ denote the probability
measure on $y$ with density proportional to
\[
\frac{e^{(\tilde g_{\eps_n}(y) - D(x,y))/\eps_n}}
     {(2\pi\eps_n)^{d/2}\sqrt{\nu(y)\mu(x)}}
\]
with respect to $\nu_n$; this is the empirical entropic plan conditioned on
$X = x$. The off-diagonal block in the linearisation is then
\[
K_{\eps_n,n}^{\nu_n}[h](x) = \int h(y)\,\pi_{x,\eps_n,n}^{\nu_n}(\diff y),
\]
and $\mathrm{Var}^{\nu_n}$ in the quadratic correction is the variance with
respect to $\pi_{x,\eps_n,n}^{\nu_n}$. The objects $K_{\eps_n,n}^{\mu_n}$
and $\mathrm{Var}^{\mu_n}$ are defined symmetrically by exchanging the roles
of $\mu$ and $\nu$. Both $K_{\eps_n,n}^{\nu_n}$ and $K_{\eps_n,n}^{\mu_n}$
are sup-norm nonexpansive, and the constant function $1$ is an eigenfunction
with eigenvalue $1$. They are therefore invertible on the quotient of the
space of bounded functions by the constants.

We now control each of the three terms in the Taylor expansion, starting with the quadratic remainder. Note that the object of interest is $(\bar g_{\eps_n} -\tilde g_{\eps_n})/\eps_n$ (resp. with $f$) explaining that we aim at controlling 
\[
\tfrac{1}{2\eps_n^2}   \mathrm{Var}_{\pi_{x,\eps_n,n}^{\nu_n}}\!\bigl(\bar g_{\eps_n} - \tilde g_{\eps_n}\bigr) = K_{\eps_n,n}^{\nu_n} \left[\left( (\id -K_{\eps_n,n}^{\nu_n} )\left[ (\bar g_{\eps_n} -\tilde g_{\eps_n})/\eps_n   \right]\right)^2\right]
\]

Remark that $K_{\eps_n,n}^{\nu_n}$ is a smooth approximation of identity and that the sequence of functions $(\bar g_{\eps_n} -\tilde g_{\eps_n})/\eps_n$ is smooth. Indeed, recall the definition of each of them, equation~\eqref{eq: defBar}. Thus,  $(\id -K_{\eps_n,n}^{\nu_n} )\left[ (\bar g_{\eps_n} -\tilde g_{\eps_n})/\eps_n \right]$  converges to zero uniformly.

To isolate the fluctuations from the term
\[
\begin{pmatrix} I & K_{\eps_n,n}^{\nu_n}\\ K_{\eps_n,n}^{\mu_n} & I\end{pmatrix}\!
\begin{pmatrix}\bar f_{\eps_n} - \tilde f_{\eps_n}\\
                \bar g_{\eps_n} - \tilde g_{\eps_n}\end{pmatrix}
\]
we apply the block inversion formula.
Note that since $K_{\eps_n,n}^{\nu_n}$ and $K_{\eps_n,n}^{\mu_n}$ are sup-norm
contractions on the quotient by the constants, the Neumann series
$\sum_{\kappa\ge 0}(K_{\eps_n,n}^{\nu_n}K_{\eps_n,n}^{\mu_n})^\kappa$ and
$\sum_{\kappa\ge 0}(K_{\eps_n,n}^{\mu_n}K_{\eps_n,n}^{\nu_n})^\kappa$
converge there in operator norm, and the block-inversion formula is valid. After using the push-though formula, the inverse reads
\[
 R_{\eps_n,n} :=
\begin{pmatrix}(I - K_{\eps_n,n}^\nu K_{\eps_n,n}^\mu)^{-1}
               & -K_{\eps_n,n}^\nu (I - K_{\eps_n,n}^\mu K_{\eps_n,n}^\nu)^{-1}\\[2pt]
               -(I - K_{\eps_n,n}^\mu K_{\eps_n,n}^\nu)^{-1} K_{\eps_n}^\mu
               & (I - K_{\eps_n,n}^\mu K_{\eps_n,n}^\nu)^{-1}\end{pmatrix}.
\]
Define similarly its population counterpart  $R_{\eps_n}$. From the algebraic identity $A^{-1} - B^{-1} = A^{-1}(B  - A)B^{-1} $, we get the decomposition
 \begin{align*}
 &\eps_n^{-1} R_{\eps_n, n} F_{\eps_n}(\tilde f_{\eps_n}, \tilde g_{\eps_n}) \\
 &= \eps_n^{-1} R_{\eps_n} F_{\eps_n}(\tilde f_{\eps_n}, \tilde g_{\eps_n})\\
 & +\eps_n^{-1}R_{\eps_n, n} \begin{pmatrix}  0 & K_{\eps_n}^{\nu}-K_{\eps_n,n}^{\nu_n}\\ K_{\eps_n}^{\mu}-K_{\eps_n,n}^{\mu_n} & 0 \end{pmatrix}
 R_{\eps_n} F_{\eps_n}(\tilde f_{\eps_n}, \tilde g_{\eps_n}).
 \end{align*}

One can rewrite 
\begin{align*}
\eps_n^{-1} F_{\eps_n}(\tilde f_{\eps_n}, \tilde g_{\eps_n})  = \log  \begin{pmatrix} \int \pi_{\eps_n}(x,y) \diff \nu_n(y) \\  \int \pi_{\eps_n}(x,y) \diff \mu_n(x) 
\end{pmatrix},
\end{align*}
where the log in understood entrywise. We know by the CLT that for fixed epsilon, the quantity inside the logarithm converges to one. 
At the first order, for $\eps_n$ decaying not too fast\footnote{At this stage, we only know that $\eps$ constant works out.}, 
\[
\log  \begin{pmatrix} \int \pi_{\eps_n}(x,y) \diff \nu_n(y) \\  \int \pi_{\eps_n}(x,y) \diff \mu_n(x) 
\end{pmatrix} \asymp \frac{1}{\sqrt{n}}\begin{pmatrix} \int \pi_{\eps_n}(x,y) \diff \mathbb{G}_n^\nu(y) \\  \int \pi_{\eps_n}(x,y) \mathbb{G}_n^\mu(x) 
\end{pmatrix}
\]
and we naturally will have to take $\eps_n$ in the appropriate regime. 
Thus, 
\begin{align*}
&\eps_n^{-1} R_{\eps_n} F_{\eps_n}(\tilde f_{\eps_n}, \tilde g_{\eps_n}) \\
&\asymp \frac{1}{\sqrt{n}} \begin{pmatrix}(I - K_{\eps_n}^\nu K_{\eps_n}^\mu)^{-1}
               & -K_{\eps_n}^\nu (I - K_{\eps_n}^\mu K_{\eps_n}^\nu)^{-1}\\[2pt]
               -(I - K_{\eps_n}^\mu K_{\eps_n}^\nu)^{-1} K_{\eps_n}^\mu
               & (I - K_{\eps_n}^\mu K_{\eps_n}^\nu)^{-1}\end{pmatrix}\begin{pmatrix} \int \pi_{\eps_n}(x,y) \diff \mathbb{G}_n^\nu(y) \\  \int \pi_{\eps_n}(x,y) \mathbb{G}_n^\mu(x) 
\end{pmatrix}
\end{align*}

Let us solely focus on the term
\[
\frac1{\sqrt{n}}(I - K_{\eps_n}^\nu K_{\eps_n}^\mu)^{-1} \int \pi_{\eps_n}(x,y) \diff \mathbb{G}_n^\nu(y).
\]
By Proposition~\ref{prop: SinkFluctKer} (which is proved by pure analytical means and does not build upon the result we are currently proving), 
\begin{align*}
&\frac1{\sqrt{n}}(I - K_{\eps_n}^\nu K_{\eps_n}^\mu)^{-1} \int \pi_{\eps_n}(x,y) \diff \mathbb{G}_n^\nu(y) \\
&\asymp \frac1{\sqrt{n}} \int  \frac{\mathcal{F}^{-1}\left[ \tfrac{
\mathcal{F}\left[\exp\left(-\frac1{\eps_n} D(x,\cdot) \right)\right](\xi)
}{4 \pi^2 \xi^\top [\nabla^2 \phi(x) ]\xi }
\right](y)}{(2\pi\eps_n)^{d/2} \eps_n \nu^{1/2}(y)\mu^{1/2}(x)  } \diff \mathbb{G}_n^\nu(y) 
\end{align*}

From there, a classical empirical process argument, noticing that 
\[
\mathcal{F}\left[\exp\left(-\frac1{\eps_n} D(x,\cdot) \right)\right](\xi) \asymp e^{ - \eps_n \xi ^\top \nabla^2 \phi^*  \xi }
\]

The change of variable $\sqrt{\eps_n} \xi  = \zeta$ in the Fourier space coupled with a classical empirical process argument, then shows that the choice of the bandwidth $n \eps_n ^{d/2}/\log(n) \to \infty$ suffices for  $\eps_n^{-1} R_{\eps_n} F_{\eps_n}(\tilde f_{\eps_n}, \tilde g_{\eps_n})$ to converge to zero uniformly.

WIth these first insights, let us now turn to
\[
\eps_n^{-1} \begin{pmatrix}  0 & K_{\eps_n}^{\nu_n}-K_{\eps_n,n}^{\nu_n}\\ K_{\eps_n}^{\mu_n}-K_{\eps_n,n}^{\mu_n} & 0 \end{pmatrix}
 R_{\eps_n} F_{\eps_n}(\tilde f_{\eps_n}, \tilde g_{\eps_n}).
\]

A quick matrix algebra computation shows that each relevant term will now have a form similar to\footnote{In certain terms, an addition $K$ operator appears. As it is smoothing, it doesn't change the picture below. } 
\[
(K_{\eps_n}^{\nu_n}-K_{\eps_n,n}^{\nu_n} )(I - K_{\eps_n}^\mu K_{\eps_n}^\nu)^{-1}  \log \int \pi_{\eps_n}(x,y) \diff \mu_n(x). 
\]
For the same reason as above, we linearize first the logarithm and shall show afterwards, that we picked the right regime. 
A direct computation then yields, 
\begin{align*}
(K_{\eps_n}^{\nu_n}-K_{\eps_n,n}^{\nu_n} )(I - K_{\eps_n}^\mu K_{\eps_n}^\nu)^{-1}  \log \int \pi_{\eps_n}(x,y) \diff \mu_n(x) \\
\asymp \frac1{n } \int \pi_{\eps_n}  \int  (I - K_{\eps_n}^\mu K_{\eps_n}^\nu)^{-1} \pi_{\eps_n} \diff \mathbb{G}_n^\mu\diff \mathbb{G}_n^\nu.
\end{align*}
This U-statistic is well behaved and admits infinitely many derivatives for each fixed $\eps_n$. Even though $R_{\eps_n, n}\eps_n$ behaves asymptotically  like\footnote{The \emph{like} in the sentence above is important, it isn't the inverse of a differential operator and there are thus no boundary conditions.} the inverse a second order differential operator, our choice of bandwitdh guarantees that the $U$-statistic multiplied by $\eps_n^{1}$ still converges to 0; yielding the sought claim.

\paragraph{Linearization by hand}

With the previous results at hand, the linearization is carried out exactly as in the proof of Theorem~\ref{thm: multivariate_limit}, up to replacing $f$ by $g$ where necessary.
For the fluctuations of $\hat f_\eps$, we get 
\[
a_n ( \hat f_{\eps_n}(x) -f_{\eps_n}(x)) + a_n K_{\eps_n}^\nu ( \hat g_{\eps_n}(x) -g_{\eps_n}(x)) =   \frac{a_n\eps_n}{\sqrt{n}}  
 \int \pi_{\eps_n}(x, \cdot) \diff \mathbb{G}_n^\nu + \oh_p(1).
\]
The other system comes naturally by swapping $(f,\mu)$ with $(g, \nu)$.

\paragraph{Final inversion}
By symmetry,  
\begin{equation}
\label{eq: FirstBlock}
\begin{pmatrix}
\id & K_{\eps_n}^{{\nu}} \\
 K_{\eps_n}^{{\mu}}& \id  \\
\end{pmatrix} 
\begin{pmatrix}
\hat f_{\eps_n} -f_{\eps_n}\\
\hat g_{\eps_n} -g_{\eps_n}
\end{pmatrix}
= - \frac{\eps_n}{\sqrt{n}} 
\begin{pmatrix}
 \int \pi_{\eps_n}(\cdot,y) \diff \mathbb{G}_n^{{\nu}}(y) \\
 \int \pi_{\eps_n}(x,\cdot) \diff \mathbb{G}_n^{{\mu}}(x) 
\end{pmatrix}
+ \oh_p(a_n^{-1}). 
\end{equation}

Using the formula for inversion of block operators 
\[
\begin{pmatrix}
\id & K_{\eps_n}^{{\nu}} \\
 K_{\eps_n}^{{\mu}}& \id  
\end{pmatrix} ^{-1} = 
\begin{pmatrix}
\id + K_{\eps_n}^{{\nu}} ( \id -  K_{\eps_n}^{{\mu}} K_{\eps_n}^{{\nu}})^{-1} K_{\eps_n}^{{\mu}}
& -K_{\eps_n}^{{\nu}} ( \id -  K_{\eps_n}^{{\mu}} K_{\eps_n}^{{\nu}})^{-1} \\
-( \id -  K_{\eps_n}^{{\mu}} K_{\eps_n}^{{\nu}})^{-1} K_{\eps_n}^{{\mu}} 
&  ( \id -  K_{\eps_n}^{{\mu}} K_{\eps_n}^{{\nu}})^{-1} 
\end{pmatrix} .
\]
We recall that the inverse operators appearing on the right hand side above exist for each $\eps>0$ as stated in 
\cite{carlier2020differential,gonzalez2022weak,gonzalez2023weak}.
However, note that the constant function 1 is an eigenfunction  with eigenvalue 1 of both $K_{\eps_n}^{\nu}$ and $K_{\eps_n}^{\mu}$, as 
\[
\int \pi_{\eps_n}(\cdot,y) \diff {\nu(y)}=1= \int \pi_{\eps_n}(x,\cdot) \diff {\mu(x)}\,.
\]
Thus, invertibility only holds for appropriate, but non constant functions, which is the case for the quantities on the right hand side of Equation~\eqref{eq: FirstBlock}.

For each $\eps_n>0$, we have a  Neumann series expansion of the inverse on the set of non-constant functions. It indeed holds
\[( \id -  K_{\eps_n}^{{\mu}} K_{\eps_n}^{{\nu}})^{-1} \\
= \sum_{\kappa=0}^\infty \big(K_{\eps_n}^{{\mu}} K_{\eps_n}^{{\nu}}\big)^\kappa.
\]
In the view of this, we get the simplification
\[
\begin{pmatrix}
\id & K_{\eps_n}^{{\nu}} \\
 K_{\eps_n}^{{\mu}}& \id  
\end{pmatrix} ^{-1} = 
\begin{pmatrix}
 ( \id -  K_{\eps_n}^{{\nu}} K_{\eps_n}^{{\mu}})^{-1} 
& -K_{\eps_n}^{{\nu}} ( \id -  K_{\eps_n}^{{\mu}} K_{\eps_n}^{{\nu}})^{-1} \\
-( \id -  K_{\eps_n}^{{\mu}} K_{\eps_n}^{{\nu}})^{-1} K_{\eps_n}^{{\mu}} 
&  ( \id -  K_{\eps_n}^{{\mu}} K_{\eps_n}^{{\nu}})^{-1} 
\end{pmatrix}. \qedhere
\]
\end{proof}

%--------------------------------------------------------------------------------------------------------------------------
%                                   Proof of Sinkhorn kernel
%---------------------------------------------------------------------------------------------------------------------------

\subsection{Proof of Proposition~\ref{prop: SinkFluctKer}}
\label{sec: SinkFluctKer}

\begin{proof}[Proof of Proposition~\ref{prop: SinkFluctKer}]
The problem is understanding
\begin{align*}
&y \mapsto( \id -  K_{\eps_n}^{{\nu}} K_{\eps_n}^{{\mu}})^{-1}  [\pi_{\eps_n}(x,y)] \\
&\qquad\qquad\qquad =y \mapsto ( \id -  K_{\eps_n}^{{\nu}} K_{\eps_n}^{{\mu}})^{-1} \left[\frac
{
\exp\left(-\frac1{\eps_n} D(x,y) + \oh(1) \right)
}{(2\pi\eps_n)^{d/2} \nu^{1/2}(y)\mu^{1/2}(x)}  \right].
\end{align*}
Recall the understanding of  $K_{\eps_n}^{{\nu}} K_{\eps_n}^{{\mu}}$ provided by Theorem~\ref{thm: LimOp}. The correct order of the spectral gap is $\Oh(\eps_n)$ so that $\eps_n ( \id -  K_{\eps_n}^{{\nu}} K_{\eps_n}^{{\mu}})^{-1}$ is a bounded operator for $\eps_n < \eps_0$, for some $\eps_0>0.$
The quantity 
\[
( \id -  K_{\eps_n}^{{\nu}} K_{\eps_n}^{{\mu}})^{-1}  [\pi_{\eps_n}(x,y)]
\]
is actually equivalent to finding $y\mapsto h_x(y)$ such that 
\[
 h_x(y) -  K_{\eps_n}^{{\nu}} K_{\eps_n}^{{\mu}}\big[h_x(\cdot)\big](y)  = \left[\frac{ \exp\left(-\frac1{\eps_n} D(x,y) + \oh(1) \right)}{(2\pi\eps_n)^{d/2} \nu^{1/2}(y)\mu^{1/2}(x)} \right].
\]
Finally, set 
$$
h_x(y) =\frac{ \tilde h_x(y)   }{ \nu^{1/2}(y) \mu^{1/2}(x)}.
$$

Then, 
\begin{align*}
&\mathcal{F} \left[K_{\eps_n}^{{\nu}} K_{\eps_n}^{{\mu}}[\tilde h_x]\right](\xi) \\
%--------------------------------
&\asymp   \int 
 \tilde h_x(y)  e ^{-i2\pi \langle \xi, y \rangle} \Bigg[\int \frac{1 }{(2\pi)^{d/2}}  \sqrt{
 \frac{
    \det[ \nabla^2 \phi_0^{*}(y+\sqrt{\eps_n}u)] \det[ \nabla^2 \phi_0^{*}(y)]} 
    { \det[ \nabla^2 \phi_0^{*}(y+\sqrt{\eps_n}u)+ \nabla^2 \phi_0^{*}(y)]}
    }
 \\
&\quad \quad \times \exp \left(-  \frac1{4} 
u^\top \nabla^2 \phi_0^{*}(y+\sqrt{\eps_n}u) u \right) 
  e ^{-i2\pi \langle \sqrt{\eps_n} \xi, u \rangle} \diff u \Bigg] \diff y\\
  %--------------------------------------------------------------------------------------------------------------------
  &\asymp   \int 
 \tilde h_x(y)  e ^{-i2\pi \langle \xi, y \rangle} \Bigg[\int \frac{1 }{(4\pi)^{d/2}}  \sqrt{
    \det[ \nabla^2 \phi_0^{*}(y)] 
    }
 \\
&\quad \quad \times \exp \left(-  \frac1{4} 
u^\top \nabla^2 \phi_0^{*}(y) u \right) 
  e ^{-i2\pi \langle \sqrt{\eps_n} \xi, u \rangle} \diff u \Bigg] \diff y \\
    %--------------------------------------------------------------------------------------------------------------------
  &\asymp   \int 
 \tilde h_x(y)  e ^{-i2\pi \langle \xi, y \rangle} e^{ -4 \pi^2 \eps_n\ \xi^\top [\nabla^2 \phi^*(y)]^{-1}\xi } \diff y.
\end{align*}

All in all, 
\begin{align*}
&\mathcal{F}[\tilde h_x](\xi)- \mathcal{F} \left[K_{\eps_n}^{{\nu}} K_{\eps_n}^{{\mu}}[\tilde h_x]\right](\xi)
\\
&   \asymp \int  \tilde h_x(y)  e ^{-i2\pi \langle \xi, y \rangle} \Bigg[1- e^{ -4 \pi^2 \eps\ \xi^\top [\nabla^2 \phi^*(y)]^{-1}\xi } \Bigg] \diff y   \\
  %---------------------------------------------------------------
  &   \asymp \int  \tilde h_x(y)  e ^{-i2\pi \langle \xi, y \rangle}  4 \pi^2 \eps_n\ \xi^\top [\nabla^2 \phi^*(y)]^{-1}\xi  \diff y 
\end{align*}

Note that the right hand side of the integral equation containing $e^{-D(x,y)/\eps}$ forces $x^*$ to be close\footnote{Recall that the divergence compares $x$ and $y$ after transporting them to the same domain via the transport map.} to $y$.
Finally, as $\nabla^2 \phi^*(y)$ is smooth because of our assumption, we can use a method similar to that of frozen coefficients in partial differential equation theory to write 
\begin{align*}
&\mathcal{F}[\tilde h_x](\xi)- \mathcal{F} \left[K_{\eps_n}^{{\nu}} K_{\eps_n}^{{\mu}}[\tilde h_x]\right](\xi)
\\
& \asymp \int  \tilde h_x(y)  e ^{-i2\pi \langle \xi, y \rangle}  4 \pi^2 \eps_n\ \xi^\top [\nabla^2 \phi^*(x^*)]^{-1}\xi  \diff y \\
&+ 
 \int  \tilde h_x(y)  e ^{-i2\pi \langle \xi, y \rangle}  4 \pi^2 \eps_n\ \xi^\top\big [   [\nabla^2 \phi^*(y)]^{-1}-[\nabla^2 \phi^*(x^*)]^{-1}\big] \xi
\end{align*} 
It remains to use the dual flatness property to get to the claim.
\end{proof}

%--------------------------------------------------------------------------------------------------------------------------
%                                   Proof of propositions for score estimation 
%---------------------------------------------------------------------------------------------------------------------------

\subsection{Proof of Propositions~\ref{prop: PopBias} and~\ref{prop: StatFluct}}
\label{sec: ProofScoreL2}

\begin{proof}[Proof of Proposition~\ref{prop: PopBias}]
From Proposition~\ref{prop:secOrdExp}, the following holds after a simple rearrangement
\begin{align*}
 &   -2f_\eps(x)/\eps
    - \log\mu(x)\\
    & = \frac{d}{2}\log(2\pi\eps) + \eps \operatorname{tr} \left( \frac{\nabla^2\mu(x)}{4\mu(x)} - \frac{1}8\big(\nabla \log \mu(x) \big)\big(\nabla \log \mu(x) \big)^\top \right)+ \oh(\eps), 
\end{align*}
Then, remark that 
\begin{align*}
\frac{\nabla^2\mu(x)}{\mu(x)} &= \exp (V)\nabla^2 (\exp (-V))
=  - \nabla^2 V +  (\nabla V) (\nabla V)^\top.
\end{align*}
Putting everything together, we have
\begin{align*}
-2f_\eps(x)/\eps - \log\mu(x) = \frac{d}{2}\log(2\pi\eps) +\frac{\eps}{4}\bigl(-\Delta V(x) + \frac{1}{2}\|\nabla V(x)\|^2\bigr)    + \oh(\eps)\,.
\end{align*}
Taking the gradient on both sides in the display above yields
\begin{align*}
    s_\eps(x) - \nabla \log\mu(x) &=  \frac{\eps}{4} (-\nabla\Delta V(x)  + \nabla^2 V(x)\nabla V(x)) + \oh(\eps)\,,
\end{align*}
which follows from the chain rule. Thus, for $\eps$ small enough, we arrive at
\begin{multline*}
    \|s_\eps- \nabla \log\mu\|^2_{L^2(\mu)} \\
    \lesssim \eps^2 \Bigl( \int \|\nabla \Delta V(x)\|^2 \dd \mu(x) + \int \| \nabla^2 V(x)\nabla \log \mu(x)\|^2 \dd \mu(x)\Bigr)\,.
\end{multline*}
The claim then follows from Assumption~\ref{ass: C}. 
\end{proof}

\begin{proof}[Proof of Proposition~\ref{prop: StatFluct}]
By Proposition~\ref{prop: fluctScore}, we have that
\begin{align*}
\E \| \hat s_{\eps_n} - s_{\eps_n} \|_{L^2(\mu)}^2 &\lesssim \E \int  \left\rVert  \frac1{\eps_n\sqrt{n}} \int y \frac{ \exp \left( - \tfrac1{2\eps_n}\|y -x  \|^2 \right)}{(2\pi\eps_n)^{d/2}
\mu^{1/2}(y)(1+\oh(1))
} \diff \mathbb{G}_n(y)\right\rVert^2\! \dd x \\
&+ \E \left\rVert K_{\eps_n} [x]  \eps_n^{-2}\left(
  \hat f_{\eps_n}(x) -
 f_{\eps_n}(x)
\right)
\right\rVert_{L^2(\mu)}^2 \\
&+ \E \left\rVert K_{\eps_n} \left[x  \eps_n^{-2}\left(
  \hat f_{\eps_n}(x) -
 f_{\eps_n}(x)
\right) \right]
\right\rVert_{L^2(\mu)}^2, 
\end{align*}
using the basic inequality $(a+b+c)^2 \le 3 (a^2 + b^2 +c^2)$.
Let us focus the first term.
Rewriting, 
we have 
\begin{multline*}
I_1 := n^{-1}\eps_n^{-d-2}\\
\times \int \left\| \int y \exp\left(-\tfrac{1}{2\eps_n} \|x-y \|^2 + \tfrac12 V(x) + \tfrac12 V(y)  \right) (1+ \oh(1)) \diff \mathbb{G}_n(y)\right\|^2 \diff \mu(x) .
\end{multline*}
One can thus fix an arbitrary compact $ K\ni 0$ containing the product of the supports and get that 
\begin{multline*}
I_1 \asymp n^{-1}\eps_n^{-d-2}  \\
\int \left\| \int  \1_{(x,y) \in K} y \exp\left(-\tfrac{1}{2\eps_n} \|x-y \|^2 + \tfrac12 V(x) + \tfrac12 V(y)  \right) (1+ \oh(1)) \diff \mathbb{G}_n(y)\right\|^2 \\ \mu(x) \diff x.
\end{multline*}
On that compact $K$, 
\[
 y\exp\left(-\tfrac{1}{2\eps_n} \|x-y \|^2 + \tfrac12 V(x) + \tfrac12 V(y)  \right) \to \frac{x}{\mu(x)},
\]
uniformly in $x,y$ as $\eps_n \to 0$.
Thus, 
\[
\int \left\| \int y \exp\left(-\tfrac{1}{2\eps_n} \|x-y \|^2 + \tfrac12 V(x) + \tfrac12 V(y)  \right) \diff \mathbb{G}_n(y)\right\|^2 \mu(x) \diff x = \Oh_p(1)  .
\]

Recalling from the proof of Theorem~\ref{thm: multivariate_limit} that
\[
\eps_n^{-1}\big( \hat{f}_{\eps_n} (x) -  {f}_{\eps_n} (x)\big) \asymp \frac{1}{2\sqrt{n}}  \left[\int \frac{ \exp\left(   - \tfrac1{2\eps_n}\|y -x  \|^2\right)}{{
(2\pi \eps_n )^{d/2}\mu^{1/2}(y)\mu^{1/2}(x)(1+\oh(1))
}
} \diff \mathbb{G}_n(y)\right],
\]
we have 
\begin{multline*}
 \left\rVert K_{\eps_n} [x]  \eps_n^{-2}\left(
  \hat f_{\eps_n}(x) -
 f_{\eps_n}(x)
\right)
\right\rVert_{L^2(\mu)}^2 \\\asymp n^{-1} \eps_n^{-d-2} \int 
\|x\|^2 \left(\int
\frac{ \exp\left(   - \tfrac1{2\eps_n}\|y -x  \|^2\right)}{
\mu^{1/2}(y)\mu^{1/2}(x)
} \diff \mathbb{G}_n(y)\right)^2
\mu(x)\diff x,
\end{multline*}
and similarly for the third term. Both terms can be treated in the exact same way as above.
\end{proof}

%--------------------------------------------------------------------------------------------------------------------------
%                                   Proof of Theorem for potentials in the one-measure case
%---------------------------------------------------------------------------------------------------------------------------
\section{Proof of Theorem~11}
\label{sec: multivariate_limit}

We start with a first lemma which is going to be fundamental to control higher order terms in the linearization.

\begin{lem} Under Assumptions~\ref{assum: Main}, assuming that $ n^{-2/d}\ll \eps_n$, when $n \to \infty$  it holds that
\label{lem: UniformConv}
\[
\left \| \frac{f_{\eps_n} - \hat f_{\eps_n} }{\eps_n}\right\|_\infty \to 0.
\]
\end{lem}
\begin{proof}

We start by remarking that $-2f_{\eps_n}/\eps_n - d \log \big( 2\pi \eps \big)/2 $ is bounded above and below for $\eps$ small enough.
Let $\omega_{\eps, K}(x)$  be such that $\sqrt{\omega_{\eps, K}(x)}  = \int_K k_\eps(x,y) \omega_{\eps, K}^{-1/2}(y)\diff y$.  
Around the boundary, the factors $\omega_{\eps, K}(x)$ make up for (part of) the mass lost by the fact that the kernel crosses the boundary. 

Recall the structure of dampened Sinkhorn updates in the population self-transport case. 
Initializing at  $ -\frac{\eps}2 \log \big( \mu (2 \pi \eps)^{d/2} \omega_{\eps, K}(x) \big) $ the next iterate reads 
\begin{align*}
 \frac{\eps}4  \log \big( \mu(x)  \omega_{\eps, K}(x)  \big)  + \frac{\eps}2 \log  \int k_\eps(x,y) \frac{\sqrt{\mu(y) }}{\sqrt{ \omega_{\eps, K}(y) }} \diff y \\
 =  \eps \log\Big( 1+ \Oh(\sqrt{\eps})\Big),
\end{align*}
where the equality follows from a careful application of the Laplace method including the boundary issues.

By a maximum principle argument, the function $\sqrt{\omega_{\eps, K}(x)}$ is bounded above and below.

Recall the fact that the empirical potential satisfies the fixed point equation
\[
\frac{\hat f_{\eps_n} (x)}{\eps_n} = - \log \frac1n\sum_{i=1}^n \frac{\exp(-\frac{\|x- X_i\|^2}{2\eps_n})}{(2\pi\eps_n)^{d/2}} \exp\left(\frac{\hat f_{\eps_n}(X_i)}{\eps_n}\right).
\]
As previously, we will linearize this system. 
Let 
\[
F_{\eps,n} ( h) := h + \eps \log \int k_{\eps}(x,y) e^{h(y)/\eps} \diff \mu_n(y). 
\]
and observe that this functional is convex.
A Taylor expansion yields, 
\[
0=F_{\eps,n} ( \hat f_\eps  ) = F_{\eps,n} ( f_{\eps_n} ) + \langle D F_{\eps,n} \vert_{ f_{\eps_n}} , \hat f_\eps -  f_{\eps_n} \rangle - \frac1{\eps} Var_{\bar f}( h ) ,
\]
 where $Var_{\bar f}$ means that the variance is taken with respect to a measure $\propto e^{\bar f} \mu(x)$ and $\bar f $ is an interpolant between  $\hat f_\eps$  and $ f_{\eps_n}$. $\langle , \rangle $ denotes the duality pairing between measures and continuous functions on the compact $K$.

A simple computation yields
\begin{align*}
 \langle D F_{\eps,n} \vert_{f_0}, h \rangle  
 &= \id + \frac{\int k_{\eps}(x,y) e^{f_0(y)/\eps} h(y) \diff \mu_n(y) }{\int k_{\eps}(x,y) e^{f_0(y)/\eps} \diff \mu_n(y)}
 \end{align*}

The general proof idea is to show that $F_{\eps,n} ( f_{\eps_n} )$ is small  and invert $D F_{\eps,n} \vert_{f_\eps}$ as one would classically do in Z-estimation.

Set 
\begin{equation}
\label{eq: operator}
D F_{\eps,n} \vert_{f_\eps} := \id + K_{\eps_n,n}.
\end{equation}

We start with $F_{\eps,n} ( f_{\eps_n} )$.  First, observe that 
\[
\E \left[\frac1n\sum_{i=1}^n \frac{\exp(-\frac{\|x- X_i\|^2}{\eps_n})}{(2\pi\eps_n)^{d/2}} \exp\left(\frac{ f_{\eps_n}(X_i)}{\eps_n}\right)  \right] = e^{-f_{\eps_n}(x)/\eps_n},
\]
which shows that there is no bias.
Now, the empirical fluctuations must be controlled, we need that
\[
x \mapsto \frac1n\sum_{i=1}^n \frac{\exp(-\frac{\|x- X_i\|^2}{\eps_n})}{(2\pi\eps_n)^{d/2} } \exp\left(\frac{ f_{\eps_n}(X_i)}{\eps_n}\right)
\]
concentrates around its expectation in the $L^\infty$ norm.
The classical proof of convergence of kernel density estimators in the uniform norm relies on the fact that the class of functions
\[
\mathcal{F}_0:=\left \{ x\mapsto \exp\left(-\frac{\|x- y\|^2}{h}\right); h >0, y \in \reals^d \right\}
\]
is a VC class. 
Therefore, it is clear that for any measure $Q$, and $h,h' \in \mathcal{F}_0$, 
\[
\E_{Y \sim Q} \left[ \left( h(Y)e^{\frac{ f_{\eps_n}(Y)}{\eps_n}}  - h'(Y)e^{\frac{ f_{\eps_n}(Y)}{\eps_n}} \right)^2\right] \le \ell^{-1}\ \E_{Y \sim Q} \left[ \left({f(Y)} - {f'(Y)} \right)^2\right], 
\]
as we have just shown that $f_{\eps_n}(Y)/\eps_n$ is upper and lower bounded. 
Because of this, we can rely on the same arguments as in \citet{einmahl2005uniform} to conclude 
that 
\[
\left \| \frac1n\sum_{i=1}^n \frac{\exp(-\frac{\|x- X_i\|^2}{\eps_n})}{(2\pi\eps_n)^{d/2}} \exp\left(\frac{ f_{\eps_n}(X_i)}{\eps_n}\right) -e^{-f_{\eps_n}(x)/\eps_n}\right\|_\infty \to 0 \quad a.s.,
\]
as $n\to \infty,  \eps_n \to 0$ and $\sqrt{n}\eps_n^{d/4}/\sqrt{\log n} \to \infty$.

This yields that 
\[
\frac{F_{\eps,n} ( f_{\eps_n} )}{\eps_n} = \oh(1)
\]

\medskip

Denote by $\bar K_n$ the same operator as in equation~\ref{eq: operator} with $\bar f$ (the interpolant) in place of $f_\eps$

We get after a simple algebraic reorganization
\begin{align}
\nonumber
\label{eq: NormSup}
&\left\| \frac{\hat f_\eps-f_\eps }{\eps}  \right\|_\infty =\\
 &\Bigg\| (\id + K_{\eps_n, n})^{-1}\left[\frac{F_{\eps,n} ( f_{\eps_n} )}{\eps_n}\right] - (\id + K_{\eps_n, n})^{-1} \bar K_n \left[  \left( (\id - \bar K_n)\left[ \tfrac{\hat f_\eps-f_\eps }{\eps} \right]  \right)^2\right]  \Bigg\|_\infty
\end{align}

Notice that $\bar K_n$ is a smooth approximation of identity which incorporates the knowledge of the boundaries. 
Note that using their definition,$\tfrac{\hat f_\eps-f_\eps }{\eps}$ is $\cC^\infty$ for each epsilon. 
For the record, let us write 
\[
\frac{f_\eps }{\eps} = \frac1\eps \frac{\int e^{f_\eps/\eps(y)} k_\eps(x, y)(x-y) \diff \mu(y) }{\int e^{f_\eps/\eps} k_\eps(x, \cdot) \diff \mu }.
\] 
The crudest estimation of the order of the derivative is $1/\sqrt{\eps}$  and similarly for the empirical version. Note however that the tilting combined with the $L^2$ expansion we have shown, improves on this rate.
At the above scale, $(\id - \bar K_n)$ converges to a first-order differential operator at best. Altogether, with the large bandwidth chosen,
\[
(\id - \bar K_n)\left[ \tfrac{\hat f_\eps-f_\eps }{\eps} \right] =\oh_p(1).
\]

The operator $ \bar K_n$ thus smoothes a sequences of functions converging to zero uniformly and the whole second summand in the norm on the r.h.s. of equation~\eqref{eq: NormSup} converges to zero uniformly.

Note that 
\[
(\id + K_{\eps_n, n})^{-1} = \frac{\id}{2} -  \frac{K_{\eps_n, n}- \id }{2} + \cdots.
\]

Note also that
\begin{align*}
&\nabla \frac{F_{\eps,n} ( f_{\eps_n} )}{\eps_n} \\
&=
 \frac{\int e^{f_\eps(x)/\eps } k_\eps(x,y) (x- y) \diff\mu(y) }{\eps \int e^{f_\eps(x)/\eps } k_\eps(x,y)  \diff\mu(y)}
 - \frac{\int e^{f_\eps(x)/\eps } k_\eps(x,y) (x- y) \diff\mu_n(y) }{ \eps \int e^{f_\eps(x)/\eps } k_\eps(x,y)  \diff\mu_n(y)}
\end{align*}
A polynomial is again a VC class, and the product of that VC class with that of the Gaussian distribution used above is again VC, so that a Talagrand type inequality ensures that 
\[
\left\| \nabla \frac{F_{\eps,n} ( f_{\eps_n} )}{\eps_n} \right \|_\infty= \Oh_p\big( \sqrt{\frac{ \log n }{ n \eps}}\big).
\]
Therefore, the sequence  $F_{\eps,n} ( f_{\eps_n} )/\eps_n$ remains locally smooth and successive iterations of  $K_{\eps_n, n}- \id $ won't alter the convergence to 0.
\end{proof}

\begin{proof}[Proof of Theorem~\ref{thm: multivariate_limit}]

The proof consists of two main steps: Linearization and a central limit theorem.

\textbf{Linearization.}
To simplify the notation, we replace $f_{\eps_n}$ by $f$ below.
We start by re-writing the fluctuations using the fixed point equation characterization as
\begin{align*}
&a_n ( \hat f(x) -f(x)) \\
&=
 a_n\eps_n \left( 
 \log \int \exp\left( \frac1{\eps_n  } f \right) k_{\eps_n}(x, \cdot) \diff \mu - 
  \log \int \exp\left( \frac1{\eps_n } \hat f\right) k_{\eps_n}(x, \cdot) \diff  \mu_n
\right)\\
&= a_n\eps_n \log\left( 
  \int \exp\left( \frac1{\eps_n  } f \right) k_{\eps_n}(x, \cdot) \diff \mu \Big/ \int \exp\left( \frac1{\eps_n  } f \right) k_{\eps_n}(x, \cdot) \diff \mu_n
\right)  \\
&+ a_n\eps_n \log\left( 
  \int \exp\left( \frac1{\eps_n  }  f \right) k_{\eps_n}(x, \cdot) \diff \mu_n \Big/ \int \exp\left( \frac1{\eps_n  }  \hat f \right) k_{\eps_n}(x, \cdot) \diff  \mu_n \right) \\
&=: C(x) + D(x)\,.
\end{align*} 

Then, 
\begin{align*}
C(x) &= - a_n\eps_n \log \left( 
\frac{
 \int \exp\left( \frac1{\eps_n  }  f \right) k_{\eps_n}(x, \cdot) \diff  \mu_n}
  {
   \int \exp\left( \frac1{\eps_n }  f\right) k_{\eps_n}(x, \cdot) \diff  \mu}
\right)\\
&=  - a_n\eps_n \log \Bigg( 
\frac{
  \int \exp\left( \frac1{\eps_n  }  f \right) k_{\eps_n}(x, \cdot) \diff  \mu  }   {
   \int \exp\left( \frac1{\eps_n }  f\right) k_{\eps_n}(x, \cdot) \diff  \mu} \\
   &\qquad\qquad\qquad+\frac1{\sqrt{n}}\frac{ \int \exp\left( \frac1{\eps_n  }  f \right) k_{\eps_n}(x, \cdot) \diff \big( \sqrt{n}(\mu_n- \mu)\big)}
  {
   \int \exp\left( \frac1{\eps_n }  f\right) k_{\eps_n}(x, \cdot) \diff  \mu}
\Bigg)\\
&= - \frac{a_n\eps_n}{\sqrt{n}}  
\frac{
 \int \exp\left( \frac1{\eps_n  }  f \right) k_{\eps_n}(x, \cdot) \diff \mathbb{G}_n }
  {
   \int \exp\left( \frac1{\eps_n }  f\right) k_{\eps_n}(x, \cdot) \diff  \mu} \\
& \qquad\qquad\qquad + \oh \left( - \frac{a_n\eps_n}{\sqrt{n}} \frac{
 \int \exp\left( \frac1{\eps_n  }  f \right) k_{\eps_n}(x, \cdot) \diff \mathbb{G}_n }
  {
   \int \exp\left( \frac1{\eps_n }  f\right) k_{\eps_n}(x, \cdot) \diff  \mu} \right) \\
&=- \frac{a_n\eps_n}{\sqrt{n}}  
 \int \pi_{\eps_n} \diff \mathbb{G}_n 
+ \oh \left( - \frac{a_n\eps_n}{\sqrt{n}} \frac{
 \int \exp\left( \frac1{\eps_n  }  f \right) k_{\eps_n}(x, \cdot) \diff \mathbb{G}_n }
  {
   \int \exp\left( \frac1{\eps_n }  f\right) k_{\eps_n}(x, \cdot) \diff  \mu}\right), 
\end{align*}
where $\mathbb{G}_n:= \sqrt{n}(\mu_n-\mu)$. 
Because of the development of the potentials as $\eps \to 0$, we get
\begin{align*}
 &  \frac{a_n\eps_n}{\sqrt{n}}  
 \int \pi_{\eps_n}(x, \cdot) \diff \mathbb{G}_n 
= 
\frac{a_n\eps_n}{\sqrt{n}} \int \frac{ \exp\left(\frac{-1}{2\eps_n} \|y -x  \|^2 \right)}{{
(2\pi\eps_n)^{d/2} \mu^{1/2}(y)\mu^{1/2}(x)(1+\oh(1))
}
} \diff \mathbb{G}_n(y).
\end{align*}
We now treat $D(x)$ by providing upper and lower bounds for it.

\begin{align*}
 &D(x) \\
 &= -a_n\eps_n \log \left( 
\frac{
  \int \exp\left( \frac1{\eps_n  } \hat f \right) k_{\eps_n}(x, \cdot) \diff  \mu_n}
  {
   \int \exp\left( \frac1{\eps_n }  f\right) k_{\eps_n}(x, \cdot) \diff  \mu_n}
\right)\\
&= -a_n\eps_n \log \left(  1 + 
\frac{
  \int \exp\left( \frac1{\eps_n  } \hat f \right) -  \exp\left( \frac1{\eps_n  }  f \right) k_{\eps_n}(x, \cdot) \diff  \mu_n}
  {
   \int \exp\left( \frac1{\eps_n }  f\right) k_{\eps_n}(x, \cdot) \diff  \mu_n}
\right)\\
&\geq a_n\eps_n 
\frac{
  \int \exp\left( \frac1{\eps_n  }  f \right) -  \exp\left( \frac1{\eps_n  } \hat f \right) k_{\eps_n}(x, \cdot) \diff \mu_n}
  {
   \int \exp\left( \frac1{\eps_n }  f\right) k_{\eps_n}(x, \cdot) \diff \mu_n} \qquad ( - \log (1+x) \geq -x) 
\\
&= a_n 
\frac{
  \int \exp\left( \frac1{\eps_n  } \hat  f \right) [ f-   \hat f ]\ k_{\eps_n}(x, \cdot) \diff \mu_n}
  {
   \int \exp\left( \frac1{\eps_n }  f\right) k_{\eps_n}(x, \cdot) \diff \mu_n} \qquad ( e^x -1 \geq x) 
\\
\end{align*}

At this stage note that applying the same type of  inequalities with reversed signs yield that 
\begin{equation}
\label{eq: UBDx}
D(x) \le a_n 
\frac{
  \int \exp\left( \frac1{\eps_n  }   f \right) [ f-   \hat f ]\ k_{\eps_n}(x, \cdot) \diff \mu_n}
  {
   \int \exp\left( \frac1{\eps_n }  \hat f\right) k_{\eps_n}(x, \cdot) \diff \mu_n}.
\end{equation}

Recalling that
\[
\mu_n = \mu + \frac{1}{\sqrt{n}} \sqrt{n } (\mu_n- \mu) =  \mu + \frac{1}{\sqrt{n}} \mathbb{G}_n,
\]
one gets 
\begin{align*}
 &D(x) \\
 &\geq  a_n 
\frac{
  \int \exp\left( \frac1{\eps_n  } \hat  f \right) [ f-   \hat f ]\ k_{\eps_n}(x, \cdot) \diff \mu_n}
  {
   \int \exp\left( \frac1{\eps_n }  f\right) k_{\eps_n}(x, \cdot) \diff \hat \mu + n^{-1/2}   \int \exp\left( \frac1{\eps_n }  f\right) k_{\eps_n}(x, \cdot) \diff \mathbb{G}_n } \\  
    &\geq  a_n 
\frac{
  \int \exp\left( \frac1{\eps_n  } \hat  f \right) [ f-   \hat f ]\ k_{\eps_n}(x, \cdot) \diff  \mu + n^{-1/2}   \int \exp\left( \frac1{\eps_n  } \hat  f \right) [ f-   \hat f ]\ k_{\eps_n}(x, \cdot) \diff \mathbb{G}_n}
  {
   \int \exp\left( \frac1{\eps_n }  f\right) k_{\eps_n}(x, \cdot) \diff \hat \mu + n^{-1/2}   \int \exp\left( \frac1{\eps_n }  f\right) k_{\eps_n}(x, \cdot) \diff \mathbb{G}_n } \\
   %---------------------------------------------
    &\geq  a_n 
\frac{
  \int \exp\left( \frac1{\eps_n  }   f \right) [ f-   \hat f ]\ k_{\eps_n}(x, \cdot) \diff  \mu +  
  \int \left[   \exp\left( \frac1{\eps_n  } \hat  f \right) -   \exp\left( \frac1{\eps_n  }   f \right)\right] [ f-   \hat f ] k_{\eps_n}(x, \cdot) \diff  \mu
  }
  {
   \int \exp\left( \frac1{\eps_n }  f\right) k_{\eps_n}(x, \cdot) \diff \hat \mu + n^{-1/2}   \int \exp\left( \frac1{\eps_n }  f\right) k_{\eps_n}(x, \cdot) \diff \mathbb{G}_n } \\
&   +   a_n 
\frac{
  n^{-1/2}   \int \exp\left( \frac1{\eps_n  } \hat  f \right) [ f-   \hat f ]\ k_{\eps_n}(x, \cdot) \diff \mathbb{G}_n}
 {
   \int \exp\left( \frac1{\eps_n }  f\right) k_{\eps_n}(x, \cdot) \diff \hat \mu + n^{-1/2}   \int \exp\left( \frac1{\eps_n }  f\right) k_{\eps_n}(x, \cdot) \diff \mathbb{G}_n }
 \end{align*}
Furthermore, 
\begin{align*}
\frac{
  \int [ f-   \hat f ]\ k_{\eps_n} \pi_\eps(x, \cdot) \diff  \mu +  
  \int \left[   \exp\left( \frac1{\eps_n  } \hat  f  - f \right) -  1\right] [ f-   \hat f ] \pi_{\eps_n}(x, \cdot) \diff  \mu
  }
  {
   1 + n^{-1/2}   \int  \pi_{\eps_n}(x, \cdot) \diff \mathbb{G}_n } \\
   \geq 
\frac{
  K_\eps [ f-   \hat f ] -  \frac1{\eps_n  } 
  \int [ f-   \hat f ]^2 \pi_{\eps_n}(x, \cdot) \diff  \mu
  }
  {
   1 + n^{-1/2}   \int  \pi_{\eps_n}(x, \cdot) \diff \mathbb{G}_n } \\
\end{align*}
Let us treat the denominator in the equation \eqref{eq: UBDx}. 

\begin{align*}
&\int \exp\left( \frac1{\eps_n }  \hat f\right) k_{\eps_n}(x, \cdot) \diff \mu_n \\
&= \exp\left(-  \frac1{\eps_n }   f\right) \Big ( 1 + n^{-1/2} \int \exp\left( \frac1{\eps_n } ( \hat f - f) \right) \pi \diff \mathbb{G}_n +  \int\left( \exp\left( \frac1{\eps_n } ( \hat f - f )\right) - 1\right) \pi_\eps \diff \mu\Big) \\
& \geq \exp\left(-  \frac1{\eps_n }   f\right) 
\Big ( 1 + n^{-1/2} \int \exp\left( \frac1{\eps_n } ( \hat f - f) \right) \pi \diff \mathbb{G}_n + \eps_n^{-1} K_{\eps_n} [\hat f - f  ]\Big)
\end{align*}
Therefore, 
\begin{align*}
D(x)  \le a_n 
\frac{ K_\eps  [ f-   \hat f ] 
+  n^{-1/2} \int [ f-   \hat f ]\ \pi_{\eps_n}(x, \cdot) \diff  \mathbb{G}_n
} {  1 + n^{-1/2} \int \exp\left( \frac1{\eps_n } ( \hat f - f) \right) \pi \diff \mathbb{G}_n + \eps_n^{-1} K_{\eps_n} [\hat f - f  ]}.
\end{align*}
We now aim at showing that the leading order of the upper and the lower bounds is $a_nK_\eps  [ f-   \hat f ] $ and that all other quantities are negligible. 
The choice $n^{1/2}\eps_n^{d/4}/ \log(n) \to \infty$, guarantees that 
\[
n^{-1/2}   \int  \pi_{\eps_n}(x, \cdot) \diff \mathbb{G}_n = n^{-1/2} \eps^{-d/4}  \int  \frac{k_{\eps_n}(x, \cdot)}{(2\pi \eps_n)^{d/4} \sqrt{\mu(x) \mu(y)}+\oh(1)} \diff \mathbb{G}_n 
\]
converges to 0.
As $\eps_n^{-1} $ increases slower than $a_n$, $ \eps_n^{-1} K_{\eps_n} [\hat f - f  ]$ goes to zero.

Ultimately, we get the linearization 
\begin{equation}
\label{eq: LinOneMeas}
a_n ( \hat f(x) -f(x)) + a_n K_\eps ( \hat f(x) -f(x)) =   \frac{a_n\eps_n}{\sqrt{n}}  
 \int \pi_{\eps_n}(x, \cdot) \diff \mathbb{G}_n  + \oh_p(1).
\end{equation}

\textbf{Central limit theorem.}
With the linearization at hand, we can then finish off with the central limit theorem.

   \textit{Step 1.}
   As it shall become important in the sequel, we first show limiting distribution results for quantities of the type
\begin{equation}
\label{eq: MainObj}
\left( \int_{\reals^d} \frac{ \exp\left(   - \tfrac1{2\eps_n}\|y -x_j  \|^2\right)}{{
\mu^{1/2}(y)
}
} \diff \mathbb{G}_n(y) \right)_{j=1, \ldots, m}
\end{equation}
up to choosing an appropriate scaling. 
Then, we need to control the limiting behavior of $(\id +  K_{\eps_n})^{-1}$ when the latter acts on a particular sequence of (random) functions.\\

To obtain the convergence in distribution of \eqref{eq: MainObj} after proper rescaling, we use the Lindeberg--Feller theorem \cite[Proposition~2.27]{van2000asymptotic}.
One remarks that  
\[
\mu \left[\eps_n^{-d/4} \frac{ \exp\left(   - \tfrac1{2\eps_n}\|y -x  \|^2\right)}{
\mu^{1/2}(y)}\right] < \infty
\]
as well as
\begin{align*}
&\mu \left[\left(\eps_n^{-d/4}  \exp\left(   - \tfrac1{2\eps_n}\|y -x  \|^2\right){
\mu^{-1/2}(y)}\right)^2\right] \\
& \qquad \qquad\qquad \qquad \qquad \qquad  \qquad \qquad = \int_{\reals^d} \eps_n^{-d/2}  \exp\left(   - \tfrac1{\eps_n}\|y -x  \|^2\right) \diff y \\
&\qquad \qquad \qquad \qquad \qquad \qquad \qquad \qquad = \pi^{d/2}.
\end{align*}
Further, 
\begin{align*}
&\mu \left[\eps_n^{-d/2} \frac{ \exp\left(   - \tfrac1{2\eps_n}\|y -x_1  \|^2\right)\exp\left(   - \tfrac1{2\eps_n}\|y -x_2  \|^2\right)}{
\mu(y)}\right] \\
& \qquad =  \exp\left( -\frac1{4\eps_n} \| x_1 - x_2 \|^2\right)\int_{\reals^d} \eps_n^{-d/2}  \exp\left(   - \tfrac1{\eps_n}\|y - \tfrac{x_1+x_2}2  \|^2\right) \diff y \\
& \qquad=  \exp\left( -\frac1{4\eps_n} \| x_1 - x_2 \|^2\right)\pi^{d/2}. 
\end{align*}
It is then clear that the limiting covariance matrix (as $\eps_n \to 0$) is well behaved. 
Setting 
\[
Y_{n,i}:= \left (\eps_n^{-d/4} \frac{ \exp\left(   - \tfrac1{2\eps_n}\|X_i -x_1  \|^2\right)}{
\sqrt{n}\mu^{1/2}(X_i)}, \ldots,\eps_n^{-d/4} \frac{ \exp\left(   - \tfrac1{2\eps_n}\|X_i -x_m  \|^2\right)}{
\sqrt{n}\mu^{1/2}(X_i)}  \right)^\top,   
\]
it is then a routine task to check that
\[
\sum_{i=1}^n \E \| Y_{n,i}\|^2 \1 \{\| Y_{n,i}\|>\delta\} \to 0, \text{ for every } \delta >0. 
\]
because of the assumption $\sqrt{n} \eps_n^{d/4}\sqrt{\log n} \to \infty$ as $n\to \infty$. Indeed, 
\begin{align*}
&\sum_{i=1}^n \E \| Y_{n,i}\|^2  \1 \{\| Y_{n,i}\|>\delta\}\\
&=
n \int \sum_{j=1}^m \eps_n^{-d/2} \frac{ \exp\left(   - \tfrac1{\eps_n}\|y -x_1  \|^2\right)}{
n }\\
&  \qquad \qquad\times
\1 \left\{\sum_{j=1}^m  \exp\left(   - \tfrac1{\eps_n}\|y -x_1  \|^2 + V(y) \right)>\delta^2\eps_n^{d/2} n \right\}
\diff y. 
\end{align*}
From this, we conclude that
\[
\sum_{i=1}^n(Y_{n,1} - \E Y_{n,1}, \ldots, Y_{n,m} - \E Y_{n,m} )^\top \dto \mathcal{N} ( 0_m, \pi^{d/2} I_m ) ,
\]
which is the sought limit for the quantity in Equation~\eqref{eq: MainObj} after rescaling by $\eps_n^{-d/4}$.

\textit{Step 2.} Even though the operator $K_{\eps_n}$ is an approximation of the identity, one cannot simply replace it by its limit. We thus now consider the questions arising from the presence of this (sequence of) inverse operator(s).
First remark that $(2+x)^{-1}$ is an analytic function of $x$ in a neighbourhood of zero. 
Thus, one can use a  series development for operators and get
\[
\big(2\id + (  K_{\eps_n} - \id )\big)^{-1} = \frac12 
\sum_{\kappa=0}^\infty (\id -  K_{\eps_n}  )^\kappa 2^{-\kappa}
\]
Looking at the above, 
we see that terms of the form
\[
(\id - K_{\eps_n})\left[
\int \pi_{\eps_n}(\cdot,y) \diff \GG_n(y)
\right](x)=:A+B, %\overset{\rm P}{\to} 0.
\]
will play a crucial role. 
As such computations will occur over and over again, observe that
\begin{multline*}
  B=  K_{\eps_n}\left[
\int \pi_{\eps_n}(\cdot,y) \diff \GG_n(y)
\right](x) \\
= \frac{1}{\mu^{1/2}(x)(1+\oh_x(1))} \int 
\frac{
e^{-\frac1{4 \eps_n} \|x-y\|^2}
}{
(4\pi\eps_n)^{d/2} \mu^{1/2}(y) (1+\oh_{y,z}(1))
} \\
\qquad \qquad\qquad \int \frac{
e^{-\frac1{\eps_n} \|x-z\|^2}
}{
(\pi\eps_n)^{d/2} 
} \ \diff z\ \diff\GG_n(y)\\
=
 \frac{1}{\mu^{1/2}(x)(1+\oh_x(1))} \int 
\frac{
e^{-\frac1{4 \eps_n} \|x-y\|^2}
}{
(4\pi\eps_n)^{d/2} \mu^{1/2}(y) (1+\oh_y(1))
}
\diff \GG_n(y) + \oh_p(1).
\end{multline*}
Its limiting distribution for fixed $x$, after multiplication by $\eps_n^{d/4}$, using the same argument as above, is a Gaussian random variable with variance equal to  $2^{-3d/2}\pi^{-d/2} \mu^{-1}(x)$. We further recall that the integrals of the $\oh(1)$ terms are actually negligible by our previous results on the potentials.

In general, 
\begin{align*}
&  K_{\eps_n}^\kappa[h](x_{\kappa+1})\\
&\qquad =(2\pi \eps_n)^{d\kappa/2} \mu^{-1/2}(x_{\kappa+1})\\
&\qquad\qquad\times 
\int \cdots \int \exp\left( -\frac1{2\eps_n} \sum_{j=2}^{k+1} \|x_{j}-x_{j-1} \|^2 \right) \left(1 + \sum_{j=1}^{\kappa+1} \oh_{x_j}(1) \right)\\
&\qquad\qquad\qquad\qquad\mu^{1/2}(x_1) h(x_1) \dd x_1 \cdots \dd x_{\kappa}.
\end{align*}
This multiple integral is a series of Gaussian convolution except for the integral with respect to $x_1$. We thus have, up to a negligible term,
\begin{align}
\label{eq: CompOp}
\nonumber
 & K_{\eps_n}^\kappa[h](x_{\kappa+1})\\
&= (2 \kappa \pi \eps_n)^{-d/2} \mu^{-1/2}(x_{\kappa+1})\int\exp\left( -\frac1{2\kappa\eps_n}  \|x_{\kappa+1}-x_{1} \|^2 \right) \mu^{1/2}(x_1) h(x_1) \dd x_1 .
\end{align}

Now, for any fixed $J$,
\begin{multline*}
\E\left[\left(\sum_{\kappa=0}^J 2^{-\kappa} (\id - K_{\eps_n}  )^\kappa [  h_n] \right)^2\right]\\
%----------------------
=  \sum_{0\le \kappa,\kappa' \le J } 2^{- \kappa-\kappa'} \E\left[ (\id -  K_{\eps_n})^\kappa [ h_n] (\id - K_{\eps_n})^{\kappa'} [  h_n] \right]\\
%----------------------
=  \sum_{0\le \kappa,\kappa' \le J } 2^{- \kappa-\kappa'} \E\left[ \sum_{\eta=0}^\kappa {\kappa \choose  \eta} ( - K_{\eps_n})^\eta [ h_n] \sum_{\eta'=0}^{\kappa'} {\kappa' \choose  \eta'} ( - K_{\eps_n})^{\eta'} [ h_n] \right]\\
%----------------------
=  \sum_{0\le \kappa,\kappa' \le J } 2^{- \kappa-\kappa'} \sum_{\eta=0}^\kappa\sum_{\eta'=0}^{\kappa'} {\kappa \choose  \eta} ( -1)^\eta ( -1)^{\eta'} {\kappa' \choose  \eta'} \E\left[ K_{\eps_n}^\eta [ h_n]  K_{\eps_n}^{\eta'} [ h_n]\right]  
\end{multline*}
By a direct computation similar to those above, we get 
\[
\E\left[ K_{\eps_n}^\eta [ h_n]  K_{\eps_n}^{\eta'}  [ h_n]\right] = \mu^{-1}(\cdot) (2\pi \eps_n)^{-d/2} ( \eta + \eta' +2)^{-d/2}(1+\oh(1)). 
\]
By the development above, setting 
\[
S_J(n) := \frac12 \eps_n^{d/4} \sum_{\kappa=0}^J 2^{-\kappa} (\id - K_{\eps_n}  )^\kappa [ h_n](x), 
\] 
one notices readily that $(S_J(n))_{J\in \NN}$
is Cauchy in $L^2$ with constants that can be made uniform in $n,\epsilon_n$ \textemdash recall that $x$ is fixed. 
For each fixed $J$, the distributional limit of the truncated series exists as a sum of finitely many random functions.
Denote by $\mathfrak{d}$ any distance metrizing convergence in distribution and by $X^J$ the limit as $n\to \infty$ of the truncated series $S_J(n)$. 
$X_J$ has a Gaussian distribution, one can thus add to it an independent Gaussian random variable $Y$ such that $X^J+Y$ has the Gaussian distribution which is the limiting distribution in Theorem 4; call a random variable with that distribution $X^\infty$. 
Therefore, by the triangle inequality, it holds 
\[
\mathfrak{d}( S_\infty(n), X^\infty) \le \mathfrak{d}( X^\infty, X^J) + \mathfrak{d}( S_J(n), X^J) +\mathfrak{d}( S_J(n), S_\infty(n)).
\]
As convergence in $L^2$ implies convergence in distribution, for all $\eps>0$, there exists $J$
such that 
\[
\mathfrak{d}( S_\infty(n), X^\infty) \le \eps/3+ \mathfrak{d}( S_J(n), X^J) + \eps/3.
\] 
Now, for $n$ sufficiently large, the middle term in the above display can be made smaller than $\eps/3$. 
Finally, recall the exact form of the asymptotically negligible terms in the linearization step above. As these terms also involve 
\[
\frac{k_{\eps_n}(x,y)}{(2\pi\eps_n)^{d/4}\sqrt{\mu(x)(1+\oh(1))}}, 
\]
applying $(\id + K_{\eps_n})^{-1}$ to these terms is valid and they remain asymptotically negligible after applying that operator. 
\end{proof}

%-------------------------------------------------------------------------------------------
%                   Theorem ---- Pointwise score estimator
%---------------------------------------------------------------------------------------------

\newpage
\section{Proof of Theorem~12}
\label{sec: Plans}

We start with a linearization result for the barycentric projections. 
\begin{prop}
\label{prop: fluctScore}
   Let us consider the exact same setting as in Theorem~\ref{thm: Plans}. Make the choice  
   $ a_n  :=\sqrt{n}\eps_n^{d/4} \eps_n^{1/2}$, with $a_n/\log (n) \to \infty$, as $n \to \infty$. Then, 
      \begin{align*}
&   a_n\eps_n\left(\nabla f_{\eps_n}(x) - \nabla \hat f_{\eps_n}(x) \right)\\
   &=
  -\frac{a_n\eps_n}{\sqrt{n}} \int (y-x) \frac{ \exp \left( - \tfrac1{2\eps_n}\|y -x  \|^2 \right)}{(2\pi\eps_n)^{d/2}
\mu^{1/2}(y)\mu^{1/2}(x)(1+\oh(1))
} \diff \mathbb{G}_n(y) 
\\&-a_n\left(
  \hat f_{\eps_n}(x) -
 f_{\eps_n}(x)
\right) K_{\eps_n} [(\id-x)](x) -  K_{\eps_n} \left[ (\id -x)\times   a_n\left(
  \hat f_{\eps_n}(\cdot) -
 f_{\eps_n}(\cdot)\right) \right](x)\\
 &+  \oh_p(1).
   \end{align*}
\end{prop}

\begin{proof}[Proof of Proposition~\ref{prop: fluctScore}]
By \citet[Proposition 2]{pooladian2021entropic}, 
we can express the gradient of the potentials as barycentric projections, so that 
\begin{align*}
a_n\eps_n&\left(\nabla f_{\eps_n}(x) - \nabla \hat f_{\eps_n}(x) \right)\\
&= a_n\eps_n\int (y-x) \diff \pi_{\eps_n}^x (y) - \int (y-x) \diff \hat\pi_{\eps_n}^x (y) \\
&= -a_n\eps_n\int (y-x) \exp\left(\frac1{\eps_n}  \left(f_{\eps_n}(x)+ f_{\eps_n}(y) - \frac12\|x-y\|^2\right)\right)\diff( \mu_n- \mu)(y)\\
&+ a_n\eps_n\int (y-x) k_{\eps_n}(x,y)\\
&\quad \times \left(\exp\left(
\frac1{\eps_n}(f_{\eps_n}(x)+ f_{\eps_n}(y))
\right)- \exp\left(
\frac1{\eps_n}( \hat f_{\eps_n}(x)+ \hat f_{\eps_n}(y)
\right) \right) \diff \mu_n(y)\\
&=: A+ B.
\end{align*}
$A$ is dealt with using the expansion for the potentials and we get
\[
A = -\frac{a_n\eps_n}{\sqrt{n}} \int (y-x) \frac{ \exp \left( - \tfrac1{2\eps_n}\|y -x  \|^2 \right)}{ (2\pi\eps_n)^{d/2}
\mu^{1/2}(y)\mu^{1/2}(x)(1+\oh(1))
} \diff \mathbb{G}_n(y).
\]
Let us now turn to $B$. It holds that
\begin{align*}
&B(x)=\\
&a_n\eps_n\int (y-x) k_{\eps_n}(x,y)\\
& \qquad \left(\exp\left(
\tfrac1{\eps_n}(f_{\eps_n}(x)+ f_{\eps_n}(y)
\right)- \exp\left(
\tfrac1{\eps_n}( \hat f_{\eps_n}(x)+ \hat f_{\eps_n}(y)
\right) \right) \diff \mu_n(y)\\
%----------------------------------------------------------
&=a_n\eps_n
\int (y-x) k_{\eps_n}(x,y)\exp\left(
\tfrac{f_{\eps_n}(x)+ f_{\eps_n}(y)}{\eps_n}
\right) \\ & \qquad\qquad  \times 
\left(1- \exp\left(
\tfrac{a_n}{\eps_na_n}(  \hat f_{\eps_n}(x)+ \hat f_{\eps_n}(y) -
 f_{\eps_n}(x)-  f_{\eps_n}(y) 
\right) \right) \diff \mu_n(y)\\
%----------------------------------------------------------
&=a_n\eps_n\int (y-x) k_{\eps_n}(x,y)\exp\left(
\frac1{\eps_n}(f_{\eps_n}(x)+ f_{\eps_n}(y)
\right) \\ 
& \qquad \qquad \times 
 \left(
\tfrac{-a_n}{\eps_na_n}(  \hat f_{\eps_n}(x)+ \hat f_{\eps_n}(y) -
 f_{\eps_n}(x)-  f_{\eps_n}(y) 
\right)  \diff \mu_n(y) \\
&+ \oh_p(1)  \\
%----------------------------------------------------------
&=
-a_n \left(  \hat f_{\eps_n}(x) - f_{\eps_n}(x)
\right)\int (y-x) k_{\eps_n}(x,y)\exp\left(
\tfrac1{\eps_n}(f_{\eps_n}(x)+ f_{\eps_n}(y)
\right)
  \diff \mu_n(y)  \\
  &+ \oh_p(1)\\
  &+\int (y-x) k_{\eps_n}(x,y)\exp\left(
\tfrac1{\eps_n}(f_{\eps_n}(x)+ f_{\eps_n}(y)
\right)
 \left(
-a_n(   \hat f_{\eps_n}(y) -  f_{\eps_n}(y) 
\right)  \diff \mu_n(y)
\\
&= -a_n\left(
  \hat f_{\eps_n}(x) -
 f_{\eps_n}(x)
\right) K_{\eps_n} [\id -x](x) -  K_{\eps_n} \left[(y-x) \ a_n\left(
  \hat f_{\eps_n}(y) -
 f_{\eps_n}(y)\right) \right](x) + \oh_p(1)\\
&
-\frac{a_n}{\sqrt{n}}(  \hat f_{\eps_n}(x) -
 f_{\eps_n}(x)
)\int (y-x) k_{\eps_n}(x,y)\exp\left(
\tfrac1{\eps_n}(f_{\eps_n}(x)+ f_{\eps_n}(y)
\right)
  \diff \mathbb{G}_n(y)\\
&+ \int (y-x) k_{\eps_n}(x,y)\exp\left(
\tfrac1{\eps_n}(f_{\eps_n}(x)+ f_{\eps_n}(y)
\right)
 \left(
-\frac{a_n}{\sqrt{n}}(   \hat f_{\eps_n}(y) -  f_{\eps_n}(y) 
\right)  \diff \mathbb{G}_n(y).
\end{align*}

The above linearization relying on $1 - e ^u = u + \Oh (u^2)$ is justifed by Lemma~\ref{lem: UniformConv}.
For the two last summands, notice first 
that while $a_n$ is chosen such that 
\[-\frac{a_n }{\sqrt{n} \eps_n }\int (y-x) k_{\eps_n}(x,y)\exp\left(
\tfrac1{\eps_n}(f_{\eps_n}(x)+ f_{\eps_n}(y)
\right)
  \diff \mathbb{G}_n(y)
  \]
  converges to a random variable, while
  $\eps_n^{-1}(  \hat f_{\eps_n}(x) -
 f_{\eps_n}(x)
) \to 0$
by Lemma~\ref{lem: UniformConv}.
By equation~\eqref{eq: LinOneMeas} the last term of the expansion reads
\[
-\tfrac{a_n\eps_n}{n} \int (y-x) \pi_{\eps_n}(x,y)
 (\id + K_{\eps_n})^{-1}[\pi_{\eps_n}(y,\cdot)](z) \diff \mathbb{G}_n(z)\diff \mathbb{G}_n(y).
\]
Note then that by our choice of bandwidth, implying that 
\[
\frac {a_n }{\sqrt{n}\eps_n^{d/4} \eps_n^{1/2}} \asymp 1,
\]
the scaling is too slow for the U-type statistics and the whole term is $\oh_p(1)$.

\end{proof}

\begin{proof}[Proof of Theorem~\ref{thm: Plans}]
The proof relies on the representation derived from Proposition~\ref{prop: fluctScore}.
The proof is split into three parts. \\
\underline{Step~1.} We start by stating the limit of 
\[
A_1(x):=\frac{a_n\eps_n}{\sqrt{n}} \int (y-x) \frac{ \exp \left( - \tfrac1{2\eps_n}\|y -x  \|^2 \right)}{(2\pi\eps_n)^{d/2}
\mu^{1/2}(y)\mu^{1/2}(x)(1+\oh(1))
} \diff \mathbb{G}_n(y). 
\]
We will invoke the Lindeberg--Feller central limit theorem.

Choose $a_n$ so that $a_n \eps_n^{1/2-d/4} /\sqrt{n} \to 1$, we remark that 
\[
\int  (y-x)(y-x)^\top \frac{ \exp \left( - \tfrac1{\eps_n}\|y -x  \|^2 \right)}{\eps_n(2\pi\eps_n)^{d/2}
\mu(y)\mu(x)(1+\oh(1))
} \diff \mu(y) = \mu^{-1}(x) 2^{-d/2}  I + \oh(1), 
\]
which is the first condition pertaining to the variance.  The tail condition also holds along the lines of the argument in the proof of Theorem~\ref{thm: multivariate_limit}. This yields, choosing $a_n$ as above, 
\begin{align*}
A_1(x) = \frac1{(2\pi)^{d/4}\eps_n^{1/2}}\int (y-x) \frac{ \exp \left( - \tfrac1{2\eps_n}\|y -x  \|^2 \right)}{\eps_n^{1/2}(2\pi\eps_n)^{d/4}
\mu^{1/2}(y)\mu^{1/2}(x)(1+\oh(1))
} \diff \mathbb{G}_n(y)\\
 \dto \mathcal{N}\big(0_d, \mu^{-1}(x) 2^{-d} \pi^{-d/2} I\big)
\end{align*}

\noindent
\underline{Step~2.}
We now study the term
\[
A_2(x):=K_{\eps_n} \left[ (\id -x) \ a_n\left(
  \hat f_{\eps_n}(\cdot) -
 f_{\eps_n}(\cdot)\right) \right](x)
\]
Observe that 
\begin{align*}
 &  A_2(x) \\
 &= -\frac{a_n\eps_n}{\sqrt{n}}\int (y-x) \pi_{\eps_n} (x, y) \\
 &\qquad\qquad  (\id +K_{\eps_n})^{-1}\left[\int  \frac{ \exp \left( - \tfrac1{2\eps_n}\|\cdot -z  \|^2 \right)}{(2\pi\eps_n)^{d/2}
\mu^{1/2}(z)\mu^{1/2}(\cdot)(1+\oh(1))
} \diff \mathbb{G}(z)
\right](y) \diff \mu(y)\\
&= -\frac{a_n\eps_n}{\sqrt{n}}\int   (\id +K_{\eps_n})^{-1}\Big[ (\id-x) \times\pi_{\eps_n} (x, \cdot)\Big](y) \\
& \qquad\qquad \qquad \int  \frac{ \exp \left( - \tfrac1{2\eps_n}\|y -z  \|^2 \right)}{(2\pi\eps_n)^{d/2}
\mu^{1/2}(z)\mu^{1/2}(y)(1+\oh(1))
} \diff \mathbb{G}(z)
\diff \mu(y),
\end{align*}
as for each $\eps_n$, $(\id +K_{\eps_n})^{-1}$ is a self-adjoint operator from $L^2(\mu)$ to $L^2(\mu)$.
Also, let us observe that 
\begin{align*}
&K_{\eps_n}\Big[ (\id-x) \times \pi_{\eps_n} (x, \cdot)\Big](y) \\
&=\int (z-x) \pi_{\eps_n}(y, z) \pi_{\eps_n}(x, z) \mu(z) \diff z\\
 &= \frac{1}{(4\pi\eps_n)^{d/2}\mu^{1/2}(x)\mu^{1/2}(y)} \int (z-x)\frac1{(\pi\eps_n)^{d/2}( 1+ \oh(1))} \\
&  \quad \quad  \quad \quad  \quad \quad  \quad \quad \quad \quad \times \exp\left( -\frac{\|x-z\|^2 + \|y-z\|^2}{2\eps_n} \right) \diff z\\
&= \frac{\exp\left( -\frac1{4\eps_n}  \|x-y \|^2 \right)}{(4\pi\eps_n)^{d/2}\mu^{1/2}(x)\mu^{1/2}(y)}\\
& \quad  \quad \quad \quad  \quad \quad \times  \int (z-x)\frac1{(\pi\eps_n)^{d/2}( 1+ \oh(1))} \exp\left( -\frac1{\eps_n} \left \|z- \frac{x+y}2\right\|^2 \right) \diff z.
\end{align*}
Therefore,
\begin{multline*}
K_{\eps_n}\Big[ (\id-x) \times \pi_{\eps_n} (x, \cdot)\Big](y) \\
=\frac{\exp\left( -\frac1{4\eps_n}  \|x-y \|^2 \right)}{(4\pi\eps_n)^{d/2}\mu^{1/2}(x)\mu^{1/2}(y)}\frac{y-x}2 (1+\oh(1)).
\end{multline*}
In full generality, for $\kappa> 1$,
\begin{align*}
&K_{\eps_n}^\kappa\Big[ (\id -x) \times \pi_{\eps_n} (x, \cdot)\Big](y) \\
&=K_{\eps_n}^{\kappa-1}\left[ \frac{\exp\left( -\frac1{4\eps_n}  \|\cdot-x \|^2 \right)}{(4\pi\eps_n)^{d/2}\mu^{1/2}(\cdot)\mu^{1/2}(x)}\frac{\cdot-x}2 (1+\oh(1))\right](y)\\
%-------------------------------------------------
& = \frac{\mu^{-1/2}(x)\mu^{-1/2}(y)}{\big(2\pi (\kappa-1)\eps_n\big)^{d/2}\big(4\pi\eps_n\big)^{d/2}} \\
& \qquad \times \int \frac{z-x}{2( 1+ \oh(1))} \exp\left( -\frac{ \|x-z\|^2}{4\eps_n} -\frac{ \|y-z\|^2}{2(\kappa-1)\eps_n}\right) \diff z\\
%-------------------------------------------------
&= \frac{\exp\left( -\frac1{2\eps_n(\kappa+1)}  \|x-y \|^2 \right)}{\big(2\pi (\kappa-1)\eps_n\big)^{d/2}\big(4\pi\eps_n\big)^{d/2}\mu^{1/2}(x)\mu^{1/2}(y)} \\
&\times\int \frac{z-x}{2( 1+ \oh(1))} \exp\left( -\frac{\kappa+1}{4\eps_n(\kappa-1)} \left \|z- \frac{(\kappa-1)x+2y}{\kappa+1}\right\|^2 \right) \diff z\\
%-------------------------------------------------
&= \frac{\exp\left( -\frac1{2\eps_n(\kappa+1)}  \|x-y \|^2 \right)}{\big(4\pi\eps_n\big)^{d/2}\mu^{1/2}(x)\mu^{1/2}(y)} 
\left(\frac{2}{\kappa+1}\right)^{d/2} \left( \frac{y-x}{\kappa+1} \right)(1+ \oh(1))\\
%-------------------------------------------------
%= \frac{\exp\left( -\frac1{2\eps_n(\kappa+1)}  \|x-y \|^2 \right)}{\big(2\pi(\kappa+1)\eps_n\big)^{d/2}\mu^{1/2}(x)\mu^{1/2}(y)} 
 %\left( x + \frac{y-x}{\kappa+1} \right)( 1+ \oh(1))
\end{align*}
 where the second equality follows from \eqref{eq: CompOp}. 
 \begin{multline*}
K_{\eps_n}^\kappa\Big[ (\id -x) \times \pi_{\eps_n} (x, \cdot)\Big](y) \\
%-------------------------------------------------
= \frac{\exp\left( -\frac1{2\eps_n(\kappa+1)}  \|x-y \|^2 \right)}{\big(2\pi(\kappa+1)\eps_n\big)^{d/2}\mu^{1/2}(x)\mu^{1/2}(y)} 
 \left( \frac{y-x}{\kappa+1} \right)( 1+ \oh(1))
\end{multline*}

As above, 
\begin{multline*}
(\id +K_{\eps_n})^{-1}\Big[ (\id -x)\times\pi_{\eps_n} (x, \cdot)\Big](y) \\
= \frac12 \sum_{\kappa=0}^\infty 2^{-\kappa}\sum_{\eta=0}^\kappa {\kappa \choose \eta} (-K_{\eps_n})^\eta \big[ (\id -x) \times\pi_{\eps_n} (x, \cdot)\Big](y),
\end{multline*}
so that 
\begin{align*}
&    A_2(x) \\
  &  =  \frac12 \sum_{\kappa=0}^\infty 2^{-\kappa}\sum_{\eta=0}^\kappa {\kappa \choose \eta} (- 1)^\eta  \iint K_{\eps_n}^\eta \big[ (\id-x) \times\pi_{\eps_n} (x, \cdot)\Big](y) \\
  & \qquad \qquad\qquad \qquad \qquad \qquad \qquad   \frac{ \exp \left( - \tfrac1{2\eps_n}\|y -z  \|^2 \right)}{(2\pi\eps_n)^{d/2}
\mu^{1/2}(z)\mu^{1/2}(y)
} \diff \mathbb{G}(z)
\diff \mu(y)\\
%--------------------------------------------------
& =  \frac12 \sum_{\kappa=0}^\infty 2^{-\kappa}\sum_{\eta=0}^\kappa {\kappa \choose \eta} (- 1)^\eta  \\
& \times\! \iint \frac{\exp\left( -\frac1{2\eps_n(\eta+1)}  \|x-y \|^2 \right)}{\big(2\pi(\eta+1)\eps_n\big)^{d/2}\mu^{1/2}(x)\mu^{1/2}(y)} 
 \left(\frac{y-x}{\eta+1} \right) \\
 &\qquad \qquad\qquad \qquad \qquad \qquad \qquad  
 \frac{ \exp \left( - \tfrac1{2\eps_n}\|y -z  \|^2 \right)}{(2\pi\eps_n)^{d/2}
\mu^{1/2}(z)\mu^{1/2}(y)
} \diff \mathbb{G}(z)
\diff \mu(y)\\
%--------------------------------------------------
& =  \frac12 \sum_{\kappa=0}^\infty 2^{-\kappa}\sum_{\eta=0}^\kappa {\kappa \choose \eta} (- 1)^\eta  \\
& \times\! \iint \frac{\exp\left( -\frac1{2\eps_n(\eta+1)}  \|x-y \|^2 \right)}{\big(2\pi(\eta+1)\eps_n\big)^{d/2}\mu^{1/2}(x)} 
 \left(  \frac{y-x}{\eta+1} \right)
 \frac{ \exp \left( - \tfrac1{2\eps_n}\|y -z  \|^2 \right)}{(2\pi\eps_n)^{d/2}
\mu^{1/2}(z)
}  \diff \mathbb{G}(z)
\diff y.
\end{align*}
Finally, using the same type of arguments as before, we get 
\begin{align*}
   & A_2(x) \\
& =  \frac12 \sum_{\kappa=0}^\infty 2^{-\kappa}\sum_{\eta=0}^\kappa {\kappa \choose \eta} (- 1)^\eta  \\
& \times \! \! \int \frac{\exp\left( -\frac1{2\eps_n(\eta+2)}  \|x-z \|^2 \right)}{\big(2\pi(\eta+2)\eps_n\big)^{d/2}\mu^{1/2}(x)\mu^{1/2}(z)} 
\frac{z-x}{\eta+2}
 \diff \mathbb{G}(z).
\end{align*}

\noindent
\underline{Step~3.} It remains to combine the elements. 
Setting 
\[
A_3(x) =a_n\left(
  \hat f_{\eps_n}(x) -
 f_{\eps_n}(x)
\right) K_{\eps_n} [\id-x](x),
\]
both steps above along with Theorem~\ref{thm: multivariate_limit}, guarantee the joint convergence, for fixed $x \in \R^d$, of the random vector 
$(A_1(x),A_2(x), A_3(x))^\top$ to a multivariate, centered Gaussian distribution. As 
\[
\sqrt{n}\, \eps_n^{d/4} \eps_n^{1/2} \Big( \nabla \hat{f}_{\eps_n} - \nabla {f}_{\eps_n} \Big) 
\]
is asymptotically equivalent to a linear combination of $(A_1(x),A_2(x), A_3(x))$, the claim follows.
\end{proof}

\section{Proof of Theorem~15}
\label{sec: LimOp}

\begin{proof}[Proof of Theorem~\ref{thm: LimOp}]
We start with preliminary developments helping towards proving the claim. We omit the negligible terms to simplify the exposition; it is however licit because of the assumptions made.
Let us use the previous expansions and the Laplace method to get
\begin{align*}
&K_{\eps_n}^{{\nu}}[h](x) \\
&= \int  \exp \left( \frac1{\eps_n} \big[ f_{\eps_n} (x) + g_{\eps_n}(y)  - \tfrac12 \| x -y \|^2 \big] \right) h(y) {\nu}(y) \diff y \\
&=  \int \frac{1}{(2\pi\eps_n)^{d/2}} \exp \left( \frac1{\eps_n} \big[ f_0(x)+  g_0(y)    - \tfrac12 \| x -y \|^2 \big] + \oh(1) \right) \\
& \qquad \qquad\times h(y)\nu(y)^{-1/2}\mu^{-1/2}(x) {\nu}(y) \diff y \\
&=  \int \frac{1}{(2\pi\eps_n)^{d/2}} \exp \left( - \frac1{2\eps_n}\big[ (y -x^*)^\top [\nabla^2 \phi_0(x)]^{-1}(y -x^*) +\oh(1)\big]+ \oh(1) \right) \\
& \qquad \qquad\times h(y)\nu(y)^{-1/2}\mu^{-1/2}(x) {\nu}(y) \diff y \\
&=  \big( \det  [\nabla^2 \phi_0(x) ]\big)^{1/2}  h(x^*)\nu(x^*)^{-1/2}\mu^{-1/2}(x) {\nu}(x^*) (1+\oh(1))\\
&=  \frac{  h(x^*){\nu}(x^*)}{\nu(x^*)} (1+\oh(1)) \\
&= h(x^*)(1+\oh(1))\,,
\end{align*}
where we recall the Monge--Ampère equation 
$
\mu(x) = \nu(x^*) \det [ \nabla^2 \phi_0(x)].
$
This type of expansions will be the cornerstone of the further developments.
Also, by symmetry, 
\begin{align*}
&K_{\eps_n}^{{\mu}}[h](y) \\
&= \int  \exp \left( \tfrac1{\eps_n} [ f_{\eps_n} (x) + g_{\eps_n}(y)  - \tfrac12 \| x -y \|^2 ] \right) h(x) {\mu}(x) \diff x \\
&\asymp \int \tfrac{1}{(2\pi\eps_n)^{d/2}} \exp \left( \tfrac1{\eps_n} [ f_0(x)+  g_0(y)    - \tfrac12 \| x -y \|^2 ]  \right)\frac{ h(x) \mu(x)}{ \nu(y)^{1/2}\mu^{1/2}(x)} \diff x \\
&\asymp \int \tfrac{1}{(2\pi\eps_n)^{d/2}} \exp \left(-  \frac1{2\eps_n} 
 \big(x-\nabla \phi_0^*(y)\big)^\top [\nabla^2 \phi_0^{*}(y)]^{-1}\big(x-\nabla \phi_0^*(y)\big)  \right) 
\\& \qquad \qquad \times \nu(y)^{-1/2}\mu^{1/2}(x) h(x) \diff x. 
\end{align*}
We now can turn to the main claim. 
\begin{align*}
&K_{\eps_n}^{{\mu}}\Big[K_{\eps_n}^{{\nu}}[h]\Big](y_2)\\
%-------------
& \asymp \frac{1}{\nu^{1/2}(y_2)} \iint \frac{1}{(2\pi\eps_n)^d}  \\
& \times \exp \left(-  \frac1{2\eps_n} 
 \big(x-\nabla \phi_0^*(y_2)\big)^\top [\nabla^2 \phi_0^{*}(y_2)]^{-1}\big(x-\nabla \phi_0^*(y_2)\big)  \right) \\
& \times \exp \left(-  \frac1{2\eps_n} 
 \big(x-\nabla \phi_0^*(y)\big)^\top [\nabla^2 \phi_0^{*}(y)]^{-1}\big(x-\nabla \phi_0^*(y)\big) \right) \diff x \ h(y) \nu^{1/2}(y)\diff y \\
%-------------
&  \asymp \frac{1}{\nu^{1/2}(y_2)} \int \frac{1}{(2\pi\eps_n)^d} 
\sqrt{ \frac{
    \det[ 2\pi \eps_n \nabla^2 \phi_0^{*}(y_2)]\det[ 2\pi \eps_n \nabla^2 \phi_0^{*}(y)]
}{
    \det [ 2\pi \eps_n \nabla^2 \phi_0^{*}(y_2)+\nabla^2 \phi_0^{*}(y)]
}
    } \ h(y)
 \\
& \times \exp \Big(-  \tfrac1{2\eps_n} 
 \big(\nabla \phi_0^*(y)-\nabla \phi_0^*(y_2)\big)^\top [\nabla^2 \phi_0^{*}(y_2) + \nabla^2\phi_0^{*}(y)]^{-1}\\
 &\qquad \qquad  \big(\nabla \phi_0^*(y)-\nabla \phi_0^*(y_2)\big)  \Big) \nu^{1/2}(y)\diff y,  
\end{align*}
where the last equivalence follows from the fact that the inner integral is a convolution of two Gaussians, one with mean $0$ and covariance matrix $\nabla^2 \phi_0^{*}(y_2)$ and the other one with mean $\nabla \phi_0^*(y)$ and covariance matrix $\nabla^2 \phi_0^{*}(y_2)$.
Further, 
\begin{align*}
&K_{\eps_n}^{{\mu}}\Big[K_{\eps_n}^{{\nu}}[h]\Big](y_2)\\
&  \asymp \frac{1}{\nu^{1/2}(y_2)} \int \frac{1}{(2\pi\eps_n)^{d/2}} 
\sqrt{ \frac{
    \det[ \nabla^2 \phi_0^{*}(y_2)]\det[  \nabla^2 \phi_0^{*}(y)]
}{
    \det [ \nabla^2 \phi_0^{*}(y_2)+\nabla^2 \phi_0^{*}(y)]
}
    } \ h(y)
 \\
& \times \exp \Big(-  \frac1{2\eps_n} 
\big(\nabla \phi_0^*(y)-\nabla \phi_0^*(y_2)\big)^\top [\nabla^2 \phi_0^{*}(y_2) + \nabla^2\phi_0^{*}(y)]^{-1}\\
& \qquad \qquad \big(\nabla \phi_0^*(y)-\nabla \phi_0^*(y_2)\big)  \Big) 
\nu^{1/2}(y)\diff y. 
\end{align*}
Thus, owing to Lipschitzianity of $\nabla^2 \phi_0^{*}$ thanks to Caffarelli's regularity theory recalled in the introduction, 
\begin{align*}
&K_{\eps_n}^{{\mu}}\Big[K_{\eps_n}^{{\nu}}[h]\Big](y_2)\\
 & \quad \asymp \frac{1}{\nu^{1/2}(y_2)} \int \frac{1 }{(2\pi\eps_n)^{d/2}} 
 \ h(y)  
 \sqrt{ \frac{
    \det[ \nabla^2 \phi_0^{*}(y_2)]\det[  \nabla^2 \phi_0^{*}(y)]
}{
    \det [ \nabla^2 \phi_0^{*}(y_2)+\nabla^2 \phi_0^{*}(y)]
}
    }\\
&\quad \times \exp \Big(-  \frac1{4\eps_n} 
\big(\nabla \phi_0^*(y)-\nabla \phi_0^*(y_2)\big)^\top [\nabla^2 \phi_0^{*}(y_2) + \oh(1)]^{-1}\\
&\qquad \qquad \big(\nabla \phi_0^*(y)-\nabla \phi_0^*(y_2)\big)  \Big) \nu^{1/2}(y) \diff y,
\end{align*}
where the $\oh(1)$ in the display above is $\Oh (\|y-y_2\|)$.
The claim follows owing to the Bregman duality, i.e., 
\begin{align*}
& \big(\nabla \phi_0^*(y)-\nabla \phi_0^*(y_2)\big)^\top \big[\nabla^2 \phi_0^*(y_2) +\oh(1)\big]^{-1}\big(\nabla \phi_0^*(y)-\nabla \phi_0^*(y_2)\big)\\
& \qquad \qquad = (y - y_2)^\top \big[\nabla^2 \phi_0^{*}(y_2) \big](y - y_2) +  \oh(\|y - y_2\|^2),
\end{align*}
which comes from the fact that both terms of the equality above are quadratic approximations of the Bregman divergence $\phi_0(y) +  \psi_0(x) - \langle x, y\rangle$ in one or the other variable.
\end{proof}

\section{Proof of Theorem~18}

\begin{proof}[Proof of Theorem~\ref{thm: twoMeasures}]

Let us decompose the proof into its two main statements.

\textit{First claim}.
The linearization of the fixed point equation has been established in Proposition~\ref{prop: Inversion}. 
Therefore, the first part of the claim of Theorem~\ref{thm: twoMeasures} follows from  plugging-in the result of Proposition~\ref{prop: SinkFluctKer} into Proposition~\ref{prop: Inversion} and invoking a Lindeberg--Feller central limit theorem just as in the one-measure case.

\textit{Second claim}.
For the second part of the claim, the reasoning is similar. There just remains to consider the impact of taking the gradient.

First remark that $x \mapsto \pi_{\eps_n}(x, y)$ is $\cC^1$ for every $\eps_n$. Note that the operator $(\id - K_{\eps_n}^\nu K_{\eps_n}^\mu)^{-1}$ applied to $\pi_{\eps_n}$ pertains to the $x$-variable only. 
Differentiating $\pi_{\eps_n}(x, y)$ in $x$ exactly, using
$\nabla_x D(x, y) = \nabla \phi_0(x) - y = x^* - y$,
\[
\nabla_x \pi_{\eps_n}(x, y) \asymp \frac{y - x^*}{\eps_n} \pi_{\eps_n}(x, y) 
\]
where the lower-order term that we can neglect come from differentiating the prefactor
$\mu(x)^{-1/2}$. We rewrite the leading factor as a $y$-derivative.
Taylor-expanding $T_0^{-1}$ around $x^*$, with
$\nabla T_0^{-1}(x^*) = (\nabla^2 \phi_0(x))^{-1}$,
\[
\nabla_y D(x, y) = T_0^{-1}(y) - x = (\nabla^2 \phi_0(x))^{-1} (y - x^*) + O(\|y - x^*\|^2),
\]
hence $y - x^* = \nabla^2 \phi_0(x) \nabla_y D(x, y) + O(\|y - x^*\|^2)$ and
\[
\nabla_x \pi_{\eps_n}(x, y) \asymp - \nabla^2 \phi_0(x) \nabla_y \pi_{\eps_n}(x, y) 
\]
On the effective support $\|y - x^*\| \sim \sqrt{\eps_n}$, the Taylor
remainder contributes a relative $\sqrt{\eps_n}$ correction to the leading
$\eps_n^{-1/2}$-sized term. The fluctuation kernel
$(\id - K^\nu_{\eps_n} K^\mu_{\eps_n})^{-1}[\pi_{\eps_n}(x, \cdot)](y)$
inherits the same identity, since the resolvent acts on $y$ and commutes
with the matrix $\nabla^2 \phi_0(x)$, which is constant in $y$. In Fourier
this becomes multiplication by $-i 2\pi \nabla^2 \phi_0(x) \xi$.  This explains the form of the kernel.

After proper scaling, the sequence of kernels is $L^2$ integrable and the conditions to apply the Linderberg--Feller central limit theorem are met.
\end{proof}

\end{document}